\documentclass[final,3p]{elsarticle}
 \usepackage{graphics}
 \usepackage{graphicx}
 \usepackage{epsfig}
\usepackage{amssymb}
 \usepackage{amsthm}
 \usepackage{lineno}
 \usepackage{amsmath}
   \numberwithin{equation}{section}
\usepackage{mathrsfs}

\newtheorem{thm}{Theorem}[section]

\newtheorem{lem}[thm]{Lemma}

\newtheorem{defn}[thm]{Definition}

\biboptions{sort&compress,square}
\begin{document}
\begin{frontmatter}
\author{Hongfeng Li}
\ead{lihf728@nenu.edu.cn}
\author{Yong Wang\corref{cor2}}
\ead{wangy581@nenu.edu.cn}
\cortext[cor2]{Corresponding author.}

\address{School of Mathematics and Statistics, Northeast Normal University,
Changchun, 130024, China}

\title{ Sub-signature operators and the Dabrowski-Sitarz-Zalecki type theorems\\ for manifolds with boundary}
\begin{abstract}
In this paper, we define the spectral Einstein functional associated with the sub-signature operator for manifolds with boundary. Motivated by the spectral Einstein functional and the sub-signature operator, we relate them to the noncommutative residue for manifolds with boundary. And we give the proof of the Dabrowski-Sitarz-Zalecki type theorems for the spectral Einstein functional associated with the sub-signature operator on 4-dimensional manifolds with boundary.
\end{abstract}
\end{frontmatter}
\section{Introduction}
\label{1}
An eminent spectral scheme that generates geometric objects on
manifolds such as residue, scalar curvature, and other scalar combinations of curvature tensors
is the small-time asymptotic expansion of the (localised) trace
of heat kernel\cite{PBG,FGV}. The theory has very rich structures both in physics and mathematics.
Recently, Dabrowski etc. \cite{DL} defined bilinear functionals of vector fields and differential forms,
the densities of which yield the  metric and spectral Einstein functionals on even-dimensional Riemannian manifolds,
and they obtained certain
values or residues of the (localised) zeta function of the Laplacian  arising from
the Mellin transform and the coefficients of this expansion.

Let $E$ be a finite-dimensional complex vector bundle over a closed compact manifold $M$
of dimension $n$, the noncommutative residue of a pseudo-differential operator
$P\in\Psi DO(E)$ can be defined by
 \begin{equation}
res(P):=(2\pi)^{-n}\int_{S^{*}M}\mathrm{Tr}(\sigma_{-n}^{P}(x,\xi))\mathrm{d}x \mathrm{d}\xi,
\end{equation}
where $S^{*}M\subset T^{*}M$ denotes the co-sphere bundle on $M$ and
$\sigma_{-n}^{P}$ is the component of order $-n$ of the complete symbol
 \begin{equation}
\sigma^{P}:=\sum_{i}\sigma_{i}^{P}
\end{equation}
of $P$, cf. \cite{Ac,Wo,Wo1,Gu},
and the linear functional $res: \Psi DO(E)\rightarrow \mathbb{C }$
is in fact the unique trace (up to multiplication
by constants) on the algebra of pseudo-differential operators $\Psi DO(E)$.
In \cite{Co1}, Connes  used the noncommutative residue to derive a conformal 4-dimensional
 Polyakov action analogy.
Connes  proved that the noncommutative residue on a compact manifold $M$ coincided with Dixmier's trace on pseudo-differential
operators of order -dim$M$ in \cite{Co2}.
And Connes claimed that the noncommutative residue of the square of the inverse of the Dirac operator was proportioned to the Einstein-Hilbert action. Kastler\cite{Ka} gave a brute-force proof of this theorem, and Kalau and Walze\cite{KW} proved
this theorem in the normal coordinates system simultaneously, which is called the Kastler-Kalau-Walze theorem now. Based on the theory of the noncommutative reside  introduced by Wodzicki, Fedosov etc.\cite{FGLS} constructed a noncommutative
residue on the algebra of classical elements in Boutet de Monvel's calculus on a compact manifold with boundary of dimension $n>2$.
With elliptic pseudo-differential operators and  noncommutative
residue, it's a natural way to study the Kastler-Kalau-Walze type theorem and
operator-theoretic explanation of the gravitational action for manifolds with boundary.\

On the other hand, Wang generalized the Connes' results to the case of manifolds with boundary in \cite{Wa1,Wa2},
and proved the Kastler-Kalau-Walze type theorem for the Dirac operator and the signature operator on lower-dimensional manifolds
with boundary \cite{Wa3}. In \cite{Wa3,Wa4}, Wang computed $\widetilde{{\rm Wres}}[\pi^+D^{-1}\circ\pi^+D^{-1}]$ and $\widetilde{{\rm Wres}}[\pi^+D^{-2}\circ\pi^+D^{-2}]$, where the two operators are symmetric, in these cases the boundary term vanished. But for $\widetilde{{\rm Wres}}[\pi^+D^{-1}\circ\pi^+D^{-3}]$, Wang got a nonvanishing boundary term \cite{WW1}, and gave a theoretical explanation for gravitational action on boundary. In other words, Wang provides a kind of method to study the Kastler-Kalau-Walze type theorem for manifolds with boundary. In \cite{WW2}, the authors got the spectral Einstein functional
 associated with Dirac operators with torsion on compact manifolds with  boundary. For lower dimensional compact Riemannian manifolds
  with  boundary, they computed the lower dimensional  residue of $\widetilde{\nabla}_{X}\widetilde{\nabla}_{Y}D_{T}^{-4}$ and
   get the Dabrowski-Sitarz-Zalecki theorems. Motivated by \cite{DL, WW2}, we define  the spectral Einstein functional associated with the sub-signature operator for manifolds with boundary, and the motivation of this paper is to compute the noncommutative residue $\widetilde{{\rm Wres}}[\pi^+(\nabla_{X}^{A}\nabla_{Y}^{A}D_{A}^{-2})\circ\pi^+(D_{A}^{-2})]$ and $\widetilde{{\rm Wres}}[\pi^+(\nabla_{X}^{A}\nabla_{Y}^{A}D_{A}^{-1})\circ\pi^+(D_{A}^{-3})]$ on 4-dimensional compact manifolds, where $D_{A}$ is the sub-signature operator.

The paper is organized in the following way. In Sec.\ref{section:2} and \ref{section:3}, we define the spectral Einstein functional associated with the sub-signature operator and get the noncommutative residue for manifolds without boundary. In Sec.\ref{section:4} and \ref{section:5}, we prove the Dabrowski-Sitarz-Zalecki type theorems for the spectral Einstein functional associated with the sub-signature operator on 4-dimensional manifolds with boundary.

\section{The spectral functional for the sub-signature operator}
\label{section:2}
\indent Firstly, we introduce some notations about the sub-signature operator. Let $M$ be an $n$-dimensional ($n\geq 3$) oriented compact Riemannian manifold with a Riemannian metric $g^{TM}$. And let $F$ be a subbundle of $TM$, $F^\bot$ be the subbundle of $TM$ orthogonal to $F$. Then we have the following orthogonal decomposition:
\begin{align}
&TM=F\bigoplus F^\bot;\nonumber\\
&g^{TM}=g^F\bigoplus g^{{F}^\bot},
\end{align}
where $g^F$ and $g^{{F}^\bot}$ are the induced metric on $F$ and $F^\bot$.\\
\indent Let $\nabla^L$ denote the Levi-Civita connection about $g^{TM}$. In the local coordinates $x_i;$ $1\leq i\leq n$ and the fixed orthonormal frame $\{e_1,\cdots,e_n\}$, the connection matrix $(\omega_{s,t})$ is defined by
\begin{align}
\nabla^L(e_1,\cdot\cdot\cdot,e_n)=(e_1,\cdot\cdot\cdot,e_n)(\omega_{s,t}).
\end{align}
\indent Let $\epsilon (e_j^*)$,~$\iota (e_j^*)$ be the exterior and interior multiplications respectively, where $e_j^*=g^{TM}(e_j,\cdot)$.
Write
\begin{equation}
\label{a3}
\widehat{c}(e_j)=\epsilon (e_j^* )+\iota
(e_j^*);~~
c(e_j)=\epsilon (e_j^* )-\iota (e_j^* ),
\end{equation}
which satisfies
\begin{align}
&\widehat{c}(e_i)\widehat{c}(e_j)+\widehat{c}(e_j)\widehat{c}(e_i)=2g^{TM}(e_i,e_j);\nonumber\\
&c(e_i)c(e_j)+c(e_j)c(e_i)=-2g^{TM}(e_i,e_j);\nonumber\\
&c(e_i)\widehat{c}(e_j)+\widehat{c}(e_j)c(e_i)=0.
\end{align}
By \cite{LW}, we have
\begin{align}
\label{a5}
\widetilde{D}&=d+\delta=\sum^n_{i=1}c(e_i)\bigg[e_i+\frac{1}{4}\sum_{s,t}\omega_{s,t}
(e_i)[\widehat{c}(e_s)\widehat{c}(e_t)
-c(e_s)c(e_t)]\bigg].
\end{align}
\indent Let $\pi^F$ (resp. $\pi^{F^\bot}$) be the orthogonal projection from $TM$ to $F$ (resp. $F^\bot$).
Set
\begin{align}
\label{a6}
&\nabla^F=\pi^F\nabla^{L}\pi^F;\nonumber\\
&\nabla^{F^\bot}=\pi^{F^\bot}\nabla^{L}\pi^{F^\bot},
\end{align}
then $\nabla^F$(resp. $\nabla^{F^\bot}$) is a Euclidean connection on $F$ (resp. ${F^\bot}$),  let $S$ be the tensor defined by
\begin{align}
\label{a7}
\nabla^{L}=\nabla^F+\nabla^{F^\bot}+S.
\end{align}
\indent Let $e_1,\cdot\cdot\cdot,e_n$ be the orthonormal basis of $TM$ and $f_1,\cdot\cdot\cdot,f_k$ be the orthonormal basis of $F^\bot$. The sub-signature operator $D_A$ acting on $\wedge^*T^*M\otimes\mathbb{C}$ is defined by
\begin{align}
\label{a8}
D_A&=d+\delta+\sum_{i=1}^{n}\sum_{\alpha=1}^{k}c(e_i)\widehat{c}(S(e_i)f_\alpha)\widehat{c}(f_\alpha)\nonumber\\
&=\sum^n_{i=1}c(e_i)\bigg[e_i+\frac{1}{4}\sum_{s,t}\omega_{s,t}
(e_i)[\widehat{c}(e_s)\widehat{c}(e_t)
-c(e_s)c(e_t)]\bigg]+\sum_{i=1}^{n}\sum_{\alpha=1}^{k}c(e_i)\widehat{c}(S(e_i)f_\alpha)\widehat{c}(f_\alpha),
\end{align}
where $c({e_{i}})$ denotes the Clifford action.

The following lemma of Dabrowski etc.'s Einstein functional plays a key role in our proof
of the Einstein functional.
Let $V$, $W$ be a pair of vector fields on a compact
Riemannian manifold $M$ of dimension $n = 2m$. Using the Laplace operator $\Delta^{-1}_{T}=D_{T}^{-2}=\Delta+E$
acting on sections of
 a vector bundle $E_0$ of rank $2^{m}$,
 the spectral functional over vector fields defined by
\begin{lem}\cite{DL}
The Einstein functional equals to
 \begin{equation}
Wres\big(\widetilde{\nabla}_{V}\widetilde{\nabla}_{W}\Delta^{-m}_{T}\big)=\frac{\upsilon_{n-1}}{6}2^{m}\int_{M}G(V,W)vol_{g}
 +\frac{\upsilon_{n-1}}{2}\int_{M}F(V,W)vol_{g}+\frac{1}{2}\int_{M}(\mathrm{tr}E)g(V,W)vol_{g},
\end{equation}
where $G(V,W)$ denotes the Einstein tensor evaluated on the two vector fields, $F(V,W)=tr(V_{a}W_{b}F_{ab})$ and
$F_{ab}$ is the curvature tensor of the connection $T$, $\mathrm{tr}E$ denotes the trace of $E$ and $\upsilon_{n-1}=\frac{2\pi^{m}}{\Gamma(m)}$.
\end{lem}
The aim of this section is to prove the following.
\begin{thm}
For the Laplace (type) operator $\Delta_{A}=D_A^2$, the Einstein functional equals to
\begin{align}\label{c1}
Wres\big({\nabla}_{X}^{A}{\nabla}_{Y}^{A}\Delta^{-m}_{A}\big)
=&\frac{2^{m+1}\pi^{m}}{6\Gamma(m)}\int_{M}\big(Ric(V,W)-\frac{1}{2}sg(V,W)\big) vol_{g}\nonumber\\
&-\int_{M}2^{2m-3}sg(V,W) vol_{g},
\end{align}
where $s$ is the scalar curvature, $A=\sum_{i=1}^{n}\sum_{\alpha=1}^{k}c(e_i)\widehat{c}(S(e_i)f_\alpha)\widehat{c}(f_\alpha)$ and $\nabla_{X}^{A}=\nabla_X^{\wedge^\ast T^\ast M}-\frac{1}{2}[c(X)A+Ac(X)]$.
\end{thm}

\begin{proof}
By (2.12) in \cite{WW2}, for any $\psi\in\Gamma(M,\wedge^*T^*M\otimes\mathbb{C})$, we let
 \begin{align}
{\nabla}_{X}^{A}\psi=&\nabla_{X}^{\wedge^\ast T^\ast M}\psi-\frac{1}{2}[c(X)A+Ac(X)]\psi\nonumber\\
=&X\psi+\sigma(X)\psi+a(X)\psi-\frac{1}{2}[c(X)A+Ac(X)]\psi\nonumber\\
=&X\psi+\overline{B}(X)\psi,
\end{align}
where
 \begin{equation}
\sigma(X)=-\frac{1}{4}\sum_{s,t} \omega_{s,t}(X) c(e_s)c(e_t),~~ a(X)=\frac{1}{4}\sum_{s,t} \omega_{s,t}(X) \widehat{c}(e_s)\widehat{c}(e_t).
\end{equation}
Let $V=\sum_{a=1}^{n}V^{a}e_{a}$, $W=\sum_{b=1}^{n}W^{b}e_{b}$, in view of that
 \begin{equation}
F(V,W)=tr(V_{a}W_{b}F_{ab})=\sum_{a,b=1}^{n}V^{a}W^{b}tr^{\wedge^\ast T^\ast M}(F_{e_{a},e_{b}}),
\end{equation}
we obtain
 \begin{align}
F_{e_{a},e_{b}}=&(e_{a}+\overline{B}(e_{a}))(e_{b}+\overline{B}(e_{b}))-(e_{b}+\overline{B}(e_{b}))(e_{a}+\overline{B}(e_{a}))
-([e_{a},e_{b}]+\overline{B}([e_{a},e_{b}]))
\nonumber\\
=&e_{a}\circ \overline{B}(e_{b})+\overline{B}(e_{a})\circ e_{b}+\overline{B}(e_{a})B(e_{b})-e_{b}\circ \overline{B}(e_{a})
-\overline{B}(e_{b})\circ e_{a}\nonumber\\
& -\overline{B}(e_{b})\overline{B}(e_{a})-\overline{B}([e_{a},e_{b}])\nonumber\\
=&\overline{B}(e_{b})\circ e_{a}+e_{a}(\overline{B}(e_{b}))+\overline{B}(e_{a})\circ e_{b}+\overline{B}(e_{a})\overline{B}(e_{b})
-\overline{B}(e_{a})\circ e_{b}-e_{b}(\overline{B}(e_{a}))\nonumber\\
& -\overline{B}(e_{b})\circ e_{a}-\overline{B}(e_{b})\overline{B}(e_{a})-\overline{B}([e_{a},e_{b}])\nonumber\\
=&e_{a}(\overline{B}(e_{b}))-e_{b}(\overline{B}(e_{a}))+\overline{B}(e_{a})\overline{B}(e_{b})
-\overline{B}(e_{b})\overline{B}(e_{a})-\overline{B}([e_{a},e_{b}]).
\end{align}
Also, straightforward computations yield
 \begin{align}\label{c2}
\mathrm{tr}^{\wedge^\ast T^\ast M}\big(e_{a}(\overline{B}(e_{b}))\big)
=&\mathrm{tr}^{\wedge^\ast T^\ast M}\Big[e_{a}\Big(-\frac{1}{4}\sum_{s,t} \omega_{s,t}(e_{b}) c(e_s)c(e_t)
+\frac{1}{4}\sum_{s,t} \omega_{s,t}(e_{b}) \widehat{c}(e_s)\widehat{c}(e_t)\nonumber\\
&-\frac{1}{2}[c(e_{b})A+Ac(e_{b})]\Big)\Big]\nonumber\\
=&\mathrm{tr}^{\wedge^\ast T^\ast M}\Big[-\frac{1}{4}\sum_{s,t} e_{a}(\omega_{s,t}(e_{b})) c(e_s)c(e_t)
+\frac{1}{4}\sum_{s,t} e_{a}(\omega_{s,t}(e_{b})) \widehat{c}(e_s)\widehat{c}(e_t)\Big]\nonumber\\
=&0,
\end{align}
where
 \begin{equation}
\omega_{s,t}(e_{b})=0~(s=t);~~\mathrm{tr}[c(e_s)c(e_t)]=\mathrm{tr}[\widehat{c}(e_s)\widehat{c}(e_t)]=0~(s\neq t);~~\mathrm{tr}[c(e_{b})A]=0,
\end{equation}
where we take the normal coordinate about $x_0$, it follows that
\begin{align}
&\mathrm{tr}^{\wedge^\ast T^\ast M}\big(\overline{B}(e_{a})\overline{B}(e_{b})-\overline{B}(e_{b})\overline{B}(e_{a})\big)(x_0)\nonumber\\
=&\mathrm{tr}^{\wedge^\ast T^\ast M}\Big[ \Big(-\frac{1}{4}\sum_{s,t} \omega_{s,t}(e_{a}) c(e_s)c(e_t)
+\frac{1}{4}\sum_{s,t} \omega_{s,t}(e_{a}) \widehat{c}(e_s)\widehat{c}(e_t)-\frac{1}{2}[c(e_{a})A+Ac(e_{a})]\Big)\nonumber\\
&\times \Big(-\frac{1}{4}\sum_{s,t} \omega_{s,t}(e_{b}) c(e_s)c(e_t)
+\frac{1}{4}\sum_{s,t} \omega_{s,t}(e_{b}) \widehat{c}(e_s)\widehat{c}(e_t)-\frac{1}{2}[c(e_{b})A+Ac(e_{b})]\Big)\Big](x_0)\nonumber\\
&-\mathrm{tr}^{\wedge^\ast T^\ast M}\Big[\Big(-\frac{1}{4}\sum_{s,t} \omega_{s,t}(e_{b}) c(e_s)c(e_t)
+\frac{1}{4}\sum_{s,t} \omega_{s,t}(e_{b}) \widehat{c}(e_s)\widehat{c}(e_t)-\frac{1}{2}[c(e_{b})A+Ac(e_{b})]\Big)\nonumber\\
&\times \Big(-\frac{1}{4}\sum_{s,t} \omega_{s,t}(e_{a}) c(e_s)c(e_t)
+\frac{1}{4}\sum_{s,t} \omega_{s,t}(e_{a}) \widehat{c}(e_s)\widehat{c}(e_t)-\frac{1}{2}[c(e_{a})A+Ac(e_{a})]\Big)\Big](x_0)\nonumber\\
=&0,
\end{align}
and
\begin{align}
&~~~~\mathrm{tr}^{\wedge^\ast T^\ast M}\big(\overline{B}([e_{a},e_{b}])\big)(x_0)\nonumber\\
=&\mathrm{tr}^{\wedge^\ast T^\ast M} \Big(\sigma([e_{a},e_{b}])+a([e_{a},e_{b}])
-\frac{1}{2}[c([e_{a},e_{b}])A+Ac([e_{a},e_{b}])\Big)(x_0)\nonumber\\
=&0.
\end{align}
Let $\Delta_{A}=\Delta+E$. By (2.17) in \cite{WWW}, we have
 \begin{align}
E=&\frac{1}{8}\sum_{ijkl}R_{ijkl}\widehat{c}(e_i)\widehat{c}(e_j)c(e_k)c(e_l)-\frac{1}{4}s-A^2-\frac{1}{4}\sum_{j}[c(e_j)A+Ac(e_j)]^2\nonumber\\
&+\frac{1}{2}[({\nabla}_{e_{j}}^{\wedge^\ast T^\ast M}A)c(e_j)-c(e_j)({\nabla}_{e_{j}}^{\wedge^\ast T^\ast M}A)],
\end{align}
and
\begin{align}\label{c3}
\mathrm{tr}^{\wedge^\ast T^\ast M}(E)=-\frac{1}{4}s\mathrm{tr}[id]=-2^{2m-2}s.
\end{align}
Summing up (\ref{c2})-(\ref{c3}) leads to the desired equality (\ref{c1}), and the proof of
the Theorem is complete.
\end{proof}

\section{The noncommutative residue for manifolds with boundary}
\label{section:3}
In this section, to define the noncommutative residue for the sub-signature operator,
 some basic facts and formulae about Boutet de Monvel's calculus can be found in Sec.2 in \cite{Wa1}.
Let $M$ be an n-dimensional compact oriented manifold with boundary $\partial M$.
Some basic facts and formulae about Boutet de Monvel's calculus are recalled as follows.

Let $$ F:L^2({\bf R}_t)\rightarrow L^2({\bf R}_v);~F(u)(v)=\int e^{-ivt}u(t)\texttt{d}t$$ denote the Fourier transformation and
$\varphi(\overline{{\bf R}^+}) =r^+\varphi({\bf R})$ (similarly define $\varphi(\overline{{\bf R}^-}$)), where $\varphi({\bf R})$
denotes the Schwartz space and
  \begin{equation}
r^{+}:C^\infty ({\bf R})\rightarrow C^\infty (\overline{{\bf R}^+});~ f\rightarrow f|\overline{{\bf R}^+};~
 \overline{{\bf R}^+}=\{x\geq0;x\in {\bf R}\}.
\end{equation}

We define $H^+=F(\varphi(\overline{{\bf R}^+}));~ H^-_0=F(\varphi(\overline{{\bf R}^-}))$ which are orthogonal to each other. We have the following
 property: $h\in H^+~(H^-_0)$ iff $h\in C^\infty({\bf R})$ which has an analytic extension to the lower (upper) complex
half-plane $\{{\rm Im}\xi<0\}~(\{{\rm Im}\xi>0\})$ such that for all nonnegative integer $l$,
 \begin{equation}
\frac{\texttt{d}^{l}h}{\texttt{d}\xi^l}(\xi)\sim\sum^{\infty}_{k=1}\frac{\texttt{d}^l}{\texttt{d}\xi^l}(\frac{c_k}{\xi^k})
\end{equation}
as $|\xi|\rightarrow +\infty,{\rm Im}\xi\leq0~({\rm Im}\xi\geq0)$.

 Let $H'$ be the space of all polynomials and $H^-=H^-_0\bigoplus H';~H=H^+\bigoplus H^-.$ Denote by $\pi^+~(\pi^-)$ respectively the
 projection on $H^+~(H^-)$. For calculations, we take $H=\widetilde{H}=\{$rational functions having no poles on the real axis$\}$ ($\widetilde{H}$
 is a dense set in the topology of $H$). Then on $\widetilde{H}$,
 \begin{equation}
\pi^+h(\xi_0)=\frac{1}{2\pi i}\lim_{u\rightarrow 0^{-}}\int_{\Gamma^+}\frac{h(\xi)}{\xi_0+iu-\xi}\texttt{d}\xi,
\end{equation}
where $\Gamma^+$ is a Jordan close curve included ${\rm Im}\xi>0$ surrounding all the singularities of $h$ in the upper half-plane and
$\xi_0\in {\bf R}$. Similarly, define $\pi^{'}$ on $\widetilde{H}$,
 \begin{equation}
\pi'h=\frac{1}{2\pi}\int_{\Gamma^+}h(\xi)\texttt{d}\xi.
\end{equation}

So, $\pi'(H^-)=0$. For $h\in H\bigcap L^1(R)$, $\pi'h=\frac{1}{2\pi}\int_{R}h(v)\texttt{d}v$ and for $h\in H^+\bigcap L^1(R)$, $\pi'h=0$.
Denote by $\mathcal{B}$ Boutet de Monvel's algebra (for details, see Section 2 of \cite{Wa1}).

An operator of order $m\in {\bf Z}$ and type $d$ is a matrix
$$\widetilde{A}=\left(\begin{array}{lcr}
  \pi^+P+G  & K  \\
   T  &  S
\end{array}\right):
\begin{array}{cc}
\   C^{\infty}(X,E_1)\\
 \   \bigoplus\\
 \   C^{\infty}(\partial{X},F_1)
\end{array}
\longrightarrow
\begin{array}{cc}
\   C^{\infty}(X,E_2)\\
\   \bigoplus\\
 \   C^{\infty}(\partial{X},F_2)
\end{array}.
$$
where $X$ is a manifold with boundary $\partial X$ and
$E_1,E_2~(F_1,F_2)$ are vector bundles over $X~(\partial X
)$.~Here,~$P:C^{\infty}_0(\Omega,\overline {E_1})\rightarrow
C^{\infty}(\Omega,\overline {E_2})$ is a classical
pseudo-differential operator of order $m$ on $\Omega$, where
$\Omega$ is an open neighborhood of $X$ and
$\overline{E_i}|X=E_i~(i=1,2)$. $P$ has an extension:
$~{\cal{E'}}(\Omega,\overline {E_1})\rightarrow
{\cal{D'}}(\Omega,\overline {E_2})$, where
${\cal{E'}}(\Omega,\overline {E_1})~({\cal{D'}}(\Omega,\overline
{E_2}))$ is the dual space of $C^{\infty}(\Omega,\overline
{E_1})~(C^{\infty}_0(\Omega,\overline {E_2}))$. Let
$e^+:C^{\infty}(X,{E_1})\rightarrow{\cal{E'}}(\Omega,\overline
{E_1})$ denote extension by zero from $X$ to $\Omega$ and
$r^+:{\cal{D'}}(\Omega,\overline{E_2})\rightarrow
{\cal{D'}}(\Omega, {E_2})$ denote the restriction from $\Omega$ to
$X$, then define
$$\pi^+P=r^+Pe^+:C^{\infty}(X,{E_1})\rightarrow {\cal{D'}}(\Omega,
{E_2}).$$

In addition, $P$ is supposed to have the
transmission property; this means that, for all $j,k,\alpha$, the
homogeneous component $p_j$ of order $j$ in the asymptotic
expansion of the
symbol $p$ of $P$ in local coordinates near the boundary satisfies:
$$\partial^k_{x_n}\partial^\alpha_{\xi'}p_j(x',0,0,+1)=
(-1)^{j-|\alpha|}\partial^k_{x_n}\partial^\alpha_{\xi'}p_j(x',0,0,-1),$$
then $\pi^+P:C^{\infty}(X,{E_1})\rightarrow C^{\infty}(X,{E_2})$
by Section 2.1 of \cite{Wa1}.

Let $M$ be a compact manifold with boundary $\partial M$. We assume that the metric $g^{M}$ on $M$ has
the following form near the boundary
 \begin{equation}
 g^{M}=\frac{1}{h(x_{n})}g^{\partial M}+\texttt{d}x _{n}^{2} ,
\end{equation}
where $g^{\partial M}$ is the metric on $\partial M$. Let $U\subset
M$ be a collar neighborhood of $\partial M$ which is diffeomorphic $\partial M\times [0,1)$. By the definition of $h(x_n)\in C^{\infty}([0,1))$
and $h(x_n)>0$, there exists $\widetilde{h}\in C^{\infty}((-\varepsilon,1))$ such that $\widetilde{h}|_{[0,1)}=h$ and $\widetilde{h}>0$ for some
sufficiently small $\varepsilon>0$. Then there exists a metric $\widehat{g}$ on $\widehat{M}=M\bigcup_{\partial M}\partial M\times
(-\varepsilon,0]$ which has the form on $U\bigcup_{\partial M}\partial M\times (-\varepsilon,0 ]$
 \begin{equation}
\widehat{g}=\frac{1}{\widetilde{h}(x_{n})}g^{\partial M}+\texttt{d}x _{n}^{2} ,
\end{equation}
such that $\widehat{g}|_{M}=g$.
We fix a metric $\widehat{g}$ on the $\widehat{M}$ such that $\widehat{g}|_{M}=g$.
Now we recall the main theorem in \cite{FGLS}.
\begin{thm}\label{th:32}{\bf(Fedosov-Golse-Leichtnam-Schrohe)}
 Let $X$ and $\partial X$ be connected, ${\rm dim}X=n\geq3$,
 $\widetilde{A}=\left(\begin{array}{lcr}\pi^+P+G &   K \\
T &  S    \end{array}\right)$ $\in \mathcal{B}$ , and denote by $p$, $b$ and $s$ the local symbols of $P,G$ and $S$ respectively.
 Define:
 \begin{align}
{\rm{\widetilde{Wres}}}(\widetilde{A})=&\int_X\int_{\bf S}{\mathrm{tr}}_E\left[p_{-n}(x,\xi)\right]\sigma(\xi)dx \nonumber\\
&+2\pi\int_ {\partial X}\int_{\bf S'}\left\{{\mathrm{tr}}_E\left[({\mathrm{tr}}b_{-n})(x',\xi')\right]+{\mathrm{tr}}
_F\left[s_{1-n}(x',\xi')\right]\right\}\sigma(\xi')dx'.
\end{align}
Then~~ a) ${\rm \widetilde{Wres}}([\widetilde{A},B])=0 $, for any
$\widetilde{A},B\in\mathcal{B}$;~~ b) It is a unique continuous trace on
$\mathcal{B}/\mathcal{B}^{-\infty}$.
\end{thm}
Let $p_{1},p_{2}$ be nonnegative integers and $p_{1}+p_{2}\leq n$,
 denote by $\sigma_{l}(\widetilde{A})$ the $l$-order symbol of an operator $\widetilde{A}$,
 an application of (3.5) and (3.6) in \cite{Wa1} shows that
\begin{defn} The spectral Einstein functional of compact manifolds with boundary is defined by
   \begin{equation}\label{}
   Ein_{n}^{\{p_{1},p_{2}\}}M:=\widetilde{Wres}[\pi^{+}(\nabla_{X}^{A}\nabla_{Y}^{A}(D_{A}^2)^{-p_{1}})
    \circ\pi^{+}(D_{A}^{-2})^{p_{2}}],
\end{equation}
where $\pi^{+}(\nabla_{X}^{A}\nabla_{Y}^{A}(D_{A}^2)^{-p_{1}})$, $\pi^{+}(D_{A}^{-2})^{p_{2}}$ are
 elements in Boutet de Monvel's algebra\cite{Wa3}.
\end{defn}
 For the sub-signature operator
 $\nabla_{X}^{A}\nabla_{Y}^{A}D_{A}^{-2}$ and $D_{A}^{-2}$,
 denote by $\sigma_{l}(\widetilde{A})$ the $l$-order symbol of an operator $\widetilde{A}$. An application of (2.1.4) in \cite{Wa1} shows that
\begin{align}
&\widetilde{Wres}[\pi^{+}(\nabla_{X}^{A}\nabla_{Y}^{A}(D_{A}^{-2})^{p_{1}})
    \circ\pi^{+}(D_{A}^{2})^{-p_{2}}]\nonumber\\
&=\int_{M}\int_{|\xi|=1}\mathrm{tr}_{\wedge^\ast T^\ast M\otimes\mathbb{C}}
  [\sigma_{-n}(\nabla_{X}^{A}\nabla_{Y}^{A}(D_{A}^2)^{-p_{1}}
  \circ (D_{A}^{2})^{-p_{2}}]\sigma(\xi)\texttt{d}x+\int_{\partial M}\Phi,
\end{align}
where
 \begin{align}\label{c4}
\Phi=&\int_{|\xi'|=1}\int_{-\infty}^{+\infty}\sum_{j,k=0}^{\infty}\sum \frac{(-i)^{|\alpha|+j+k+1}}{\alpha!(j+k+1)!}
\mathrm{tr}_{\wedge^\ast T^\ast M\otimes\mathbb{C}}[\partial_{x_{n}}^{j}\partial_{\xi'}^{\alpha}\partial_{\xi_{n}}^{k}\sigma_{r}^{+}
(\nabla_{X}^{A}\nabla_{Y}^{A}(D_{A}^{2})^{-p_{1}})(x',0,\xi',\xi_{n})\nonumber\\
&\times\partial_{x_{n}}^{\alpha}\partial_{\xi_{n}}^{j+1}\partial_{x_{n}}^{k}\sigma_{l}((D_{A}^2)^{-p_{2}})(x',0,\xi',\xi_{n})]
\texttt{d}\xi_{n}\sigma(\xi')\texttt{d}x' ,
\end{align}
and the sum is taken over $r-k-|\alpha|+\ell-j-1=-n,r\leq-p_{1},\ell\leq-p_{2}$.\\

For the sub-signature operator
 $\nabla_{X}^{A}\nabla_{Y}^{A}D_{A}^{-1}$ and $D_{A}^{-3}$,
 similarly we have
\begin{align}
&\widetilde{Wres}[\pi^{+}(\nabla_{X}^{A}\nabla_{Y}^{A}(D_{A}^{-1})^{p_{1}}
\circ\pi^{+}(D_{A}^{3})^{-p_{2}}]\nonumber\\
&=\int_{M}\int_{|\xi|=1}\mathrm{tr}_{\wedge^\ast T^\ast M\otimes\mathbb{C}}
  [\sigma_{-n}(\nabla_{X}^{A}\nabla_{Y}^{A}(D_{A})^{-p_{1}}
  \circ (D_{A}^{3})^{-p_{2}})]\sigma(\xi)\texttt{d}x+\int_{\partial M}\widetilde{\Phi},
\end{align}
where
 \begin{align}\label{c5}
\widetilde{\Phi}=&\int_{|\xi'|=1}\int_{-\infty}^{+\infty}\sum_{j,k=0}^{\infty}\sum \frac{(-i)^{|\alpha|+j+k+1}}{\alpha!(j+k+1)!}
\mathrm{tr}_{\wedge^\ast T^\ast M\otimes\mathbb{C}}[\partial_{x_{n}}^{j}\partial_{\xi'}^{\alpha}\partial_{\xi_{n}}^{k}\sigma_{r}^{+}
(\nabla_{X}^{A}\nabla_{Y}^{A}(D_{A})^{-p_{1}})(x',0,\xi',\xi_{n})\nonumber\\
&\times\partial_{x_{n}}^{\alpha}\partial_{\xi_{n}}^{j+1}\partial_{x_{n}}^{k}\sigma_{l}((D_{A}^{3})^{-p_{2}})(x',0,\xi',\xi_{n})]
\texttt{d}\xi_{n}\sigma(\xi')\texttt{d}x' ,
\end{align}
and the sum is taken over $r-k-|\alpha|+\ell-j-1=-n,r\leq-p_{1},\ell\leq-p_{2}$.

 \section{The residue for the sub-signature operator $\nabla_{X}^{A}\nabla_{Y}^{A}D_{A}^{-2}$ and $D_{A}^{-2}$ }
\label{section:4}
In this section, we compute the spectral Einstein functional for 4-dimension compact manifolds with boundary and get a
Dabrowski-Sitarz-Zalecki type theorem in this case.
 We will consider $D_{A}^2$.
Since $[\sigma_{-4}(\nabla_{X}^{A}\nabla_{Y}^{A}D_{A}^{-2}
  \circ D_{A}^{-2}]|_{M}$ has
the same expression as $[\sigma_{-4}(\nabla_{X}^{A}\nabla_{Y}^{A}D_{A}^{-2}
  \circ D_{A}^{-2}]|_{M}$ in the case of manifolds without boundary,
so locally we can use Theorem 2.2 to compute the first term.
\begin{thm}
 Let M be a 4-dimensional compact manifold without boundary and $\nabla^{A}$ be an orthogonal
connection. Then we get the spectral Einstein functional associated to $\nabla_{X}^{A}\nabla_{Y}^{A}D_{A}^{-2}$
and $D_{A}^{-2}$ on compact manifolds without boundary
 \begin{align}
&Wres[\sigma_{-4}(\nabla_{X}^{A}\nabla_{Y}^{A}D_{A}^{-2}
  \circ D_{A}^{-2})]\nonumber\\
=&\frac{4\pi^{2}}{3}\int_{M}\Big(Ric(X,Y)-\frac{1}{2}sg(X,Y)\Big) vol_{g}-2\int_{M}sg(X,Y) vol_{g},
\end{align}
where $s$ is the scalar curvature.
\end{thm}

So we only need to compute $\int_{\partial M}\Phi$. By (\ref{a8}), we have
\begin{align}
\sigma_{1}(D_{A})=&\sqrt{-1}c(\xi);\\
\sigma_{0}(D_{A})=&\frac{1}{4}\sum_{s,t,i}\omega_{s,t}(e_i)c(e_i)\widehat{c}(e_s)\widehat{c}(e_t)-\frac{1}{4}\sum_{s,t,i}\omega_{s,t}(e_i)c(e_i)c(e_s)c(e_t)+A,
\end{align}
where $A=\sum_{i=1}^{n}\sum_{\alpha=1}^{k}c(e_i)\widehat{c}(S(e_i)f_\alpha)\widehat{c}(f_\alpha)$.\\

We define $\nabla_X^{\wedge^\ast T^\ast M}:=X+\frac{1}{4}\sum_{ij}\langle\nabla_X^L{e_i},e_j\rangle c(e_i)c(e_j)-\frac{1}{4}\sum_{ij}\langle\nabla_X^L{e_i},e_j\rangle \widehat{c}(e_i)\widehat{c}(e_j)$, which is a connection on ${\wedge^\ast T^\ast M}$. Set
\begin{equation}
B(X)=\frac{1}{4}\sum_{ij}\langle\nabla_X^L{e_i},e_j\rangle c(e_i)c(e_j)-\frac{1}{4}\sum_{ij}\langle\nabla_X^L{e_i},e_j\rangle \widehat{c}(e_i)\widehat{c}(e_j).
\end{equation}

Let $\nabla_{X}^{A}=X+B(X)-\frac{1}{2}[c(X)A+Ac(X)]$ and $\nabla_{Y}^{A}=Y+B(Y)-\frac{1}{2}[c(Y)A+Ac(Y)]$, we obtain
\begin{align}
\nabla_{X}^{A}\nabla_{Y}^{A}&=\bigg[X+B(X)-\frac{1}{2}[c(X)A+Ac(X)]\bigg]\bigg[Y+B(Y)-\frac{1}{2}[c(Y)A+Ac(Y)]\bigg]\nonumber\\
 &=XY+X[B(Y)]+B(Y)X-\frac{1}{2}X[c(Y)A+Ac(Y)]-\frac{1}{2}[c(Y)A+Ac(Y)]X\nonumber\\
 &~~~~+B(X)Y+B(X)B(Y)-\frac{1}{2}B(X)[c(Y)A+Ac(Y)]-\frac{1}{2}[c(X)A+Ac(X)]Y\nonumber\\
 &~~~~-\frac{1}{2}[c(X)A+Ac(X)]B(Y)+\frac{1}{4}[c(X)A+Ac(X)][c(Y)A+Ac(Y)],
\end{align}
where
$X=\sum_{j=1}^nX_j\partial_{x_j}, Y=\sum_{l=1}^nY_l\partial_{x_l}$.\\

Let $g^{ij}=g(dx_{i},dx_{j})$, $\xi=\sum_{j}\xi_{j}dx_{j}$ and $\nabla^L_{\partial_{i}}\partial_{j}=\sum_{k}\Gamma_{ij}^{k}\partial_{k}$, we denote that
\begin{eqnarray}
&&\sigma_{i}=-\frac{1}{4}\sum_{s,t}\omega_{s,t}
(e_i)c(e_s)c(e_t)
;~~~a_{i}=\frac{1}{4}\sum_{s,t}\omega_{s,t}
(e_i)\widehat{c}(e_s)\widehat{c}(e_t);\nonumber\\
&&\xi^{j}=g^{ij}\xi_{i};~~~~\Gamma^{k}=g^{ij}\Gamma_{ij}^{k};~~~~\sigma^{j}=g^{ij}\sigma_{i};
~~~~a^{j}=g^{ij}a_{i}.\nonumber\\
\end{eqnarray}

Then we have the following lemmas.
\begin{lem}\label{lem3} The following identities hold:
\begin{align}
 \sigma_{0}(\nabla_{X}^{A}\nabla_{Y}^{A})=&X[B(Y)]+B(X)B(Y)
-\frac{1}{2}[c(X)A+Ac(X)]B(Y)-\frac{1}{2}X[c(Y)A+Ac(Y)]\nonumber\\
&-\frac{1}{2}B(X)[c(Y)A+Ac(Y)]+\frac{1}{4}[c(X)A+Ac(X)][c(Y)A+Ac(Y)];\\
\sigma_{1}(\nabla_{X}^{A}\nabla_{Y}^{A})
=&\sum_{j,l=1}^nX_j\frac{\partial_{Y_l}}{\partial_{X_j}}\sqrt{-1}\xi_l
+\sum_jB(Y)X_j\sqrt{-1}\xi_j+\sum_lB(X)Y_l\sqrt{-1}\xi_l\nonumber\\
&-\frac{1}{2}\sum_j [c(Y)A+Ac(Y)] X_j\sqrt{-1}\xi_j-\frac{1}{2}\sum_l [c(X)A+Ac(X)]Y_l\sqrt{-1} \xi_l;\\
\sigma_{2}(\nabla_{X}^{A}\nabla_{Y}^{A})=&-\sum_{j,l=1}^nX_jY_l\xi_j\xi_l.
\end{align}
\end{lem}
Hence by Lemma 2.1 in \cite{Wa3}, we have
 \begin{lem}\label{le:31}
The symbol of the sub-signature operator
\begin{align}
\sigma_{-1}(D_{A}^{-1})&=\frac{\sqrt{-1}c(\xi)}{|\xi|^{2}}; \\
\sigma_{-2}(D_{A}^{-1})&=\frac{c(\xi)\sigma_{0}(D_{A})c(\xi)}{|\xi|^{4}}+\frac{c(\xi)}{|\xi|^{6}}\sum_{j}c(\texttt{d}x_{j})
\Big[\partial_{x_{j}}(c(\xi))|\xi|^{2}-c(\xi)\partial_{x_{j}}(|\xi|^{2})\Big].
\end{align}
\end{lem}
By (8a) and (11a) in \cite{Ka}, we get
\begin{lem} The following identities hold:
\begin{align}
\sigma_{-2}(D_{A}^{-2})=&|\xi|^{-2};\\
\sigma_{-3}(D_{A}^{-2})=&-\sqrt{-1}|\xi|^{-4}\xi_k\Big(\Gamma^k-2\sigma^k-2a^k-c(\partial_i)A-Ac(\partial_i)\Big)\nonumber\\
&-\sqrt{-1}|\xi|^{-6}2\xi^j\xi_\alpha\xi_\beta\partial_jg^{\alpha\beta}.
\end{align}
\end{lem}

 By  Lemma 4.2, Lemma 4.4 and $\sigma(p_{1}\circ p_{2})=\sum_{\alpha}\frac{1}{\alpha!}\partial^{\alpha}_{\xi}[\sigma(p_{1})]
 D_x^{\alpha}[\sigma(p_{2})]$, we obtain
\begin{lem} The following identities hold:
\begin{align}
\sigma_{0}(\nabla_{X}^{A}\nabla_{Y}^{A}D_{A}^{-2})=&
-\sum_{j,l=1}^nX_jY_l\xi_j\xi_l|\xi|^{-2};\\
\sigma_{-1}(\nabla_{X}^{A}\nabla_{Y}^{A}D_{A}^{-2})=&
\sigma_{2}(\nabla_{X}^{A}\nabla_{Y}^{A})\sigma_{-3}(D_{A}^{-2})
+\sigma_{1}(\nabla_{X}^{A}\nabla_{Y}^{A})\sigma_{-2}(D_{A}^{-2})\nonumber\\
&+\sum_{j=1}^{n}\partial_{\xi_{j}}\big[\sigma_{2}(\nabla_{X}^{A}\nabla_{Y}^{A})\big]
D_{x_{j}}\big[\sigma_{-2}(D_{A}^{-2})\big].
\end{align}
\end{lem}

\indent Since $\Phi$ is a global form on $\partial M$, so for any fixed point $x_0\in\partial M$, we choose the normal coordinates
$U$ of $x_0$ in $\partial M$ (not in $M$) and compute $\Phi(x_0)$ in the coordinates $\widetilde{U}=U\times [0,1)\subset M$ and the
metric $\frac{1}{h(x_n)}g^{\partial M}+dx_n^2.$ The dual metric of $g^M$ on $\widetilde{U}$ is ${h(x_n)}g^{\partial M}+dx_n^2.$  Write
$g^M_{ij}=g^M(\frac{\partial}{\partial x_i},\frac{\partial}{\partial x_j});~ g_M^{ij}=g^M(dx_i,dx_j)$, then

\begin{equation}\label{1}
[g^M_{ij}]= \left[\begin{array}{lcr}
  \frac{1}{h(x_n)}[g_{ij}^{\partial M}]  & 0  \\
   0  &  1
\end{array}\right];\\~~~
[g_M^{ij}]= \left[\begin{array}{lcr}
  h(x_n)[g^{ij}_{\partial M}]  & 0  \\
   0  &  1
\end{array}\right],
\end{equation}
and\\
\begin{equation}\label{2}
\partial_{x_s}g_{ij}^{\partial M}(x_0)=0, 1\leq i,j\leq n-1; ~~~g_{ij}^M(x_0)=\delta_{ij}.
\end{equation}

\indent Let $\{\widetilde{e_1},\cdot\cdot\cdot,\widetilde{e_n}\}$ be an orthonormal frame field in $U$ about $g^{\partial M}$ which is parallel along geodesics and $\widetilde{e_i}=\frac{\partial}{\partial x_i}(x_0)$, then $\{e_1=\sqrt{h(x_n)}\widetilde{e_1},\cdot\cdot\cdot,e_{n-1}=\sqrt{h(x_n)}\widetilde{e_{n-1}},e_n=\frac{\partial}{\partial x_n}\}$ is the orthonormal frame field in $\widetilde{U}$ about $g^M$. Locally ${\wedge^\ast T^\ast M}|\widetilde{U}\cong \widetilde{U}\times\wedge^*_C(n)$. Let $\{f_1,\cdot\cdot\cdot,f_n\}$ be the orthonormal basis of $\wedge^*_C(n)$. Take a spin frame field $\sigma:\widetilde{U}\rightarrow Spin(M)$ such that $\pi\sigma=\{e_1,\cdot\cdot\cdot,e_n\}$ where $\pi:Spin(M)\rightarrow O(M)$ is a double covering, then $\{[\sigma,f_i],1\leq i\leq4\}$ is an orthonormal frame of ${\wedge^\ast T^\ast M}|_{\widetilde{U}}$. In the following, since the global form $\Phi$ is independent of the choice of the local frame, so we can compute $tr_{\wedge^\ast T^\ast M}$ in the frame $\{[\sigma,f_i],1\leq i\leq4\}$. Let $\{\widehat{e_1},\cdot\cdot\cdot,\widehat{e_n}\}$ be the canonical basis of $R^n$ and $c(\widehat{e_i})\in H om(\wedge^*_C(n),\wedge^*_C(n)$ be the Clifford action. By \cite{Y}, then
\begin{align}\label{3}
c(e_i)=[\sigma,c(\widehat{e_i})];~~~c(e_i)[\sigma,f_i]=[\sigma,c(\widehat{e_i})f_i];~~~\frac{\partial}{\partial x_i}=[\sigma,\frac{\partial}{\partial x_i}],
\end{align}
then we have $\frac{\partial}{\partial x_i}c(e_i)=0$ in the above frame.
\begin{lem}{\rm \cite{Wa3}}\label{le:32}
With the metric $g^{TM}$ on $M$ near the boundary
\begin{eqnarray}\label{a9}
\partial_{x_j}(|\xi|_{g^M}^2)(x_0)&=\left\{
       \begin{array}{c}
        0,  ~~~~~~~~~~ ~~~~~~~~~~ ~~~~~~~~~~~~~{\rm if }~j<n, \\[2pt]
       h'(0)|\xi'|^{2}_{g^{\partial M}},~~~~~~~~~~~~~~~~~~~~{\rm if }~j=n,
       \end{array}
    \right.
\end{eqnarray}
\begin{eqnarray}
\partial_{x_j}[c(\xi)](x_0)&=\left\{
       \begin{array}{c}
      0,  ~~~~~~~~~~ ~~~~~~~~~~ ~~~~~~~~~~~~~{\rm if }~j<n,\\[2pt]
\partial_{x_n}(c(\xi'))(x_{0}), ~~~~~~~~~~~~~~~~~{\rm if }~j=n,
       \end{array}
    \right.
\end{eqnarray}
where $\xi=\xi'+\xi_{n}dx_{n}$.
\end{lem}

Now we  need to compute $\int_{\partial M} \Phi$. When $n=4$, then ${\rm tr}_{\wedge^*T^*M\otimes\mathbb{C}}[{\rm \texttt{id}}]=2^n=16$, the sum is taken over $
r+l-k-j-|\alpha|=-3,~~r\leq 0,~~l\leq-2,$ then we have the following five cases:\\

\noindent{\bf Case (a)~(I)}~$r=0,~l=-2,~k=j=0,~|\alpha|=1$.\\

\noindent By (\ref{c4}), we get
\begin{equation}
\label{b24}
\Phi_1=-\int_{|\xi'|=1}\int^{+\infty}_{-\infty}\sum_{|\alpha|=1}
\mathrm{tr}[\partial^\alpha_{\xi'}\pi^+_{\xi_n}\sigma_{0}(\nabla_{X}^{A}\nabla_{Y}^{A}D_{A}^{-2})\times
 \partial^\alpha_{x'}\partial_{\xi_n}\sigma_{-2}(D_{A}^{-2})](x_0)d\xi_n\sigma(\xi')dx'.
\end{equation}
By Lemma 2.2 in \cite{Wa3}, for $i<n$,
\begin{equation}
\label{b25}
\partial_{x_i}\sigma_{-2}(D_{A}^{-2})(x_0)=
\partial_{x_i}(|\xi|^{-2})(x_0)=
-\frac{\partial_{x_i}(|\xi|^{2})(x_0)}{|\xi|^4}=0,
\end{equation}
 so $\Phi_1=0$.\\

\noindent{\bf Case (a)~(II)}~$r=0,~l=-2,~k=|\alpha|=0,~j=1$.\\

\noindent By (\ref{c4}), we get
\begin{equation}
\label{b26}
\Phi_2=-\frac{1}{2}\int_{|\xi'|=1}\int^{+\infty}_{-\infty}
\mathrm{tr}[\partial_{x_n}\pi^+_{\xi_n}\sigma_{0}(\nabla_{X}^{A}\nabla_{Y}^{A}D_{A}^{-2})\times
\partial_{\xi_n}^2\sigma_{-2}(D_{A}^{-2})](x_0)d\xi_n\sigma(\xi')dx'.
\end{equation}
By Lemma 4.4, we have
\begin{eqnarray}\label{4}
\partial_{\xi_n}^2\sigma_{-2}(D_{A}^{-2})(x_0)|_{|\xi'|=1}=\partial_{\xi_n}^2(|\xi|^{-2})(x_0)|_{|\xi'|=1}=\frac{6\xi_n^2-2}{(1+\xi_n^2)^3}.
\end{eqnarray}
It follows that
\begin{align}\label{5}
\partial_{x_n}\sigma_{0}(\nabla_{X}^{A}\nabla_{Y}^{A}D_{A}^{-2})(x_0)|_{|\xi'|=1}
=&\partial_{x_n}(-\sum_{j,l=1}^nX_jY_l\xi_j\xi_l|\xi|^{-2})(x_0)|_{|\xi'|=1}\nonumber\\
=&\frac{1}{(1+\xi_n^2)^2}\sum_{j,l=1}^nX_jY_l\xi_j\xi_lh'(0).
\end{align}
By integrating formula, we obtain
\begin{align}\label{6}
\pi^+_{\xi_n}\partial_{x_n}\sigma_{0}(\nabla_{X}^{A}\nabla_{Y}^{A}D_{A}^{-2})(x_0)|_{|\xi'|=1}
=&\partial_{x_n}\pi^+_{\xi_n}\sigma_{0}(\nabla_{X}^{A}\nabla_{Y}^{A}D_{A}^{-2})(x_0)|_{|\xi'|=1}\nonumber\\
=&-\frac{i\xi_n}{4(\xi_n-i)^2}\sum_{j,l=1}^{n-1}X_jY_l\xi_j\xi_lh'(0)+\frac{2-i\xi_n}{4(\xi_n-i)^2}X_nY_nh'(0)\nonumber\\
&-\frac{i}{4(\xi_n-i)^2}\sum_{j=1}^{n-1}X_jY_n\xi_j-\frac{i}{4(\xi_n-i)^2}\sum_{l=1}^{n-1}X_nY_l\xi_l.
\end{align}
From (\ref{4}) and (\ref{6}), we obtain
\begin{align}\label{7}
&\mathrm{tr} [\partial_{x_n}\pi^+_{\xi_n}\sigma_{0}(\nabla_{X}^{A}\nabla_{Y}^{A}D_{A}^{-2})\times
\partial_{\xi_n}^2\sigma_{-2}(D_{A}^{-2})](x_0)|_{|\xi'|=1}\nonumber\\
=&8\frac{-3i\xi_n^3+i\xi_n}{(\xi_n-i)^5(\xi_n+i)^3}\sum_{j,l=1}^{n-1}X_jY_l\xi_j\xi_lh'(0)
+8\frac{-3i\xi_n^3+6\xi_n^2+i\xi_n-2}{(\xi_n-i)^5(\xi_n+i)^3}X_nY_nh'(0)\nonumber\\
&+8\frac{(1-3\xi_n^2)i}{(\xi_n-i)^5(\xi_n+i)^3}\sum_{j=1}^{n-1}X_jY_n\xi_j
+8\frac{(1-3\xi_n^2)i}{(\xi_n-i)^5(\xi_n+i)^3}\sum_{l=1}^{n-1}X_nY_l\xi_l.
\end{align}
We note that $i<n,~\int_{|\xi'|=1}\{\xi_{i_{1}}\xi_{i_{2}}\cdots\xi_{i_{2q+1}}\}\sigma(\xi')=0$,
so we omit some items that have no contribution for computing $\Phi_2$. Therefore, we get
\begin{align}\label{8}
\Phi_2=&-\frac{1}{2}\int_{|\xi'|=1}\int^{+\infty}_{-\infty}\bigg\{8\frac{-3i\xi_n^3+i\xi_n}{(\xi_n-i)^5(\xi_n+i)^3}\sum_{j,l=1}^{n-1}X_jY_l\xi_j\xi_lh'(0)\nonumber\\
&+8\frac{-3i\xi_n^3+6\xi_n^2+i\xi_n-2}{(\xi_n-i)^5(\xi_n+i)^3}X_nY_nh'(0)\bigg\}d\xi_n\sigma(\xi')dx'\nonumber\\
=&-4\sum_{j,l=1}^{n-1}X_jY_lh'(0)\int_{|\xi'|=1}\int_{\Gamma^{+}}\frac{-3i\xi_n^3+i\xi_n}{(\xi_n-i)^5(\xi_n+i)^3}\xi_j\xi_ld\xi_{n}\sigma(\xi')dx'\nonumber\\
&-4X_nY_nh'(0)\Omega_3\int_{\Gamma^{+}}\frac{-3i\xi_n^3+6\xi_n^2+i\xi_n-2}{(\xi_n-i)^5(\xi_n+i)^3}d\xi_{n}dx'\nonumber\\
=&-4\sum_{j,l=1}^{n-1}X_jY_lh'(0)\frac{4\pi}{3}\frac{2\pi i}{4!}\left[\frac{-3i\xi_n^3+i\xi_n}{(\xi_n+i)^3}\right]^{(4)}\bigg|_{\xi_n=i}dx'\nonumber\\
&-4X_nY_nh'(0)\Omega_3\frac{2\pi i}{4!}\left[\frac{-3i\xi_n^3+6\xi_n^2+i\xi_n-2}{(\xi_n+i)^3}\right]^{(4)}\bigg|_{\xi_n=i}dx'\nonumber\\
=&-\frac{\pi^2}{2}\sum_{j=1}^{n-1}X_jY_jh'(0)dx'-\frac{15}{8}X_nY_nh'(0)\pi\Omega_3dx',
\end{align}
where ${\rm \Omega_{3}}$ is the canonical volume of $S^{2}.$\\

\noindent {\bf Case (a)~(III)}~$r=0,~l=-2,~j=|\alpha|=0,~k=1$.\\

\noindent By (\ref{c4}), we get
\begin{align}\label{9}
\Phi_3&=-\frac{1}{2}\int_{|\xi'|=1}\int^{+\infty}_{-\infty}
\mathrm{tr} [\partial_{\xi_n}\pi^+_{\xi_n}\sigma_{0}(\nabla_{X}^{A}\nabla_{Y}^{A}D_{A}^{-2})\times
\partial_{\xi_n}\partial_{x_n}\sigma_{-2}(D_{A}^{-2})](x_0)d\xi_n\sigma(\xi')dx'\nonumber\\
&=\frac{1}{2}\int_{|\xi'|=1}\int^{+\infty}_{-\infty}
\mathrm{tr}[\partial_{\xi_n}^2\pi^+_{\xi_n}\sigma_{0}(\nabla_{X}^{A}\nabla_{Y}^{A}D_{A}^{-2})\times
\partial_{x_n}\sigma_{-2}(D_{A}^{-2})](x_0)d\xi_n\sigma(\xi')dx'.
\end{align}
 By Lemma 4.4, we have
\begin{eqnarray}\label{10}
\partial_{x_n}\sigma_{-2}(D_{A}^{-2})(x_0)|_{|\xi'|=1}
=-\frac{h'(0)}{(1+\xi_n^2)^2}.
\end{eqnarray}
An easy calculation gives
\begin{align}\label{11}
\pi^+_{\xi_n}\sigma_{0}(\nabla_{X}^{A}\nabla_{Y}^{A}D_{A}^{-2})(x_0)|_{|\xi'|=1}
=&\frac{i}{2(\xi_n-i)}\sum_{j,l=1}^{n-1}X_jY_l\xi_j\xi_l-\frac{1}{2(\xi_n-i)}X_nY_n\nonumber\\
&-\frac{1}{2(\xi_n-i)}\sum_{j=1}^{n-1}X_jY_n\xi_j-\frac{1}{2(\xi_n-i)}\sum_{l=1}^{n-1}X_nY_l\xi_l.
\end{align}
Also, straightforward computations yield
\begin{align}\label{111}
\partial_{\xi_n}^2\pi^+_{\xi_n}\sigma_{0}(\nabla_{X}^{A}\nabla_{Y}^{A}D_{A}^{-2})(x_0)|_{|\xi'|=1}
=\frac{i}{(\xi_n-i)^3}\sum_{j,l=1}^{n-1}X_jY_l\xi_j\xi_l-\frac{1}{(\xi_n-i)^3}X_nY_n.
\end{align}
From (\ref{10}) and (\ref{111}), we obtain
\begin{align}\label{12}
&\mathrm{tr}[\partial_{\xi_n}^2\pi^+_{\xi_n}\sigma_{0}(\nabla_{X}^{A}\nabla_{Y}^{A}D_{A}^{-2})\times
\partial_{x_n}\sigma_{-2}(D_{A}^{-2})](x_0)|_{|\xi'|=1}\nonumber\\
=&\frac{-16h'(0)i}{(\xi_n-i)^5(\xi_n+i)^2}\sum_{j,l=1}^{n-1}X_jY_l\xi_j\xi_l+\frac{16h'(0)}{(\xi_n-i)^5(\xi_n+i)^2}X_nY_n.
\end{align}
Therefore, we get
\begin{align}\label{13}
\Phi_3=&\frac{1}{2}\int_{|\xi'|=1}\int^{+\infty}_{-\infty}
\bigg(\frac{-16h'(0)i}{(\xi_n-i)^5(\xi_n+i)^2}
\sum_{j,l=1}^{n-1}X_jY_l\xi_j\xi_l+\frac{16h'(0)}{(\xi_n-i)^5(\xi_n+i)^2}X_nY_n\bigg)d\xi_n\sigma(\xi')dx'\nonumber\\
=&-8\sum_{j,l=1}^{n-1}X_jY_lh'(0)\int_{|\xi'|=1}\int_{\Gamma^{+}}\frac{i}{(\xi_n-i)^5(\xi_n+i)^2}\xi_j\xi_ld\xi_{n}\sigma(\xi')dx'\nonumber\\
&+8X_nY_nh'(0)\Omega_3\int_{\Gamma^{+}}\frac{1}{(\xi_n-i)^5(\xi_n+i)^2}d\xi_{n}dx'\nonumber\\
=&-8\sum_{j,l=1}^{n-1}X_jY_lh'(0)\frac{4\pi}{3}\frac{2\pi i}{4!}\left[\frac{i}{(\xi_n+i)^2}\right]^{(4)}
\bigg|_{\xi_n=i}dx'+8X_nY_nh'(0)\Omega_3\frac{2\pi i}{4!}\left[\frac{1}{(\xi_n+i)^2}\right]^{(4)}\bigg|_{\xi_n=i}dx'\nonumber\\
=&-\frac{5\pi^2}{3}\sum_{j=1}^{n-1}X_jY_jh'(0)dx'-\frac{5i}{4}X_nY_nh'(0)\pi\Omega_3dx'.
\end{align}

\noindent{\bf Case (b)}~$r=0,~l=-3,~k=j=|\alpha|=0$.\\

\noindent By (\ref{c4}), we get
\begin{align}\label{14}
\Phi_4&=-i\int_{|\xi'|=1}\int^{+\infty}_{-\infty}\mathrm{tr} [\pi^+_{\xi_n}
\sigma_{0}(\nabla_{X}^{A}\nabla_{Y}^{A}D_{A}^{-2})\times
\partial_{\xi_n}\sigma_{-3}(D_{A}^{-2})](x_0)d\xi_n\sigma(\xi')dx'\nonumber\\
&=i\int_{|\xi'|=1}\int^{+\infty}_{-\infty}\mathrm{tr} [\partial_{\xi_n}\pi^+_{\xi_n}
\sigma_{0}(\nabla_{X}^{A}\nabla_{Y}^{A}D_{A}^{-2})\times
\sigma_{-3}(D_{A}^{-2})](x_0)d\xi_n\sigma(\xi')dx'.
\end{align}
 By Lemma 4.4, we have
\begin{align}\label{15}
\sigma_{-3}(D_{A}^{-2})(x_0)|_{|\xi'|=1}
=&-\frac{i}{(1+\xi_n^2)^2}\bigg(\frac{1}{2}h'(0)\sum_{k<n}\xi_k[c(e_k)c(e_n)-\widehat{c}(e_k)\widehat{c}(e_n)]\nonumber\\
&-\sum_{k<n}\xi_k[c(\partial_k)A+Ac(\partial_k)]+\frac{3}{2}h'(0)\xi_n\bigg)
-\frac{2ih'(0)\xi_n}{(1+\xi_n^2)^3}.
\end{align}
An easy calculation gives
\begin{align}\label{16}
\partial_{\xi_n}\pi^+_{\xi_n}\sigma_{0}(\nabla_{X}^{A}\nabla_{Y}^{A}D_{A}^{-2})(x_0)|_{|\xi'|=1}
=&-\frac{i}{2(\xi_n-i)^2}\sum_{j,l=1}^{n-1}X_jY_l\xi_j\xi_l-\frac{1}{2(\xi_n-i)^2}X_nY_n\nonumber\\
&+\frac{1}{2(\xi_n-i)^2}\sum_{j=1}^{n-1}X_jY_n\xi_j+\frac{1}{2(\xi_n-i)^2}\sum_{l=1}^{n-1}X_nY_l\xi_l.
\end{align}
We note that $i<n,~\int_{|\xi'|=1}\{\xi_{i_{1}}\xi_{i_{2}}\cdots\xi_{i_{2q+1}}\}\sigma(\xi')=0$, so we omit some items that have no contribution for computing $\Phi_4$.
Then, we have
\begin{align}\label{17}
&\mathrm{tr}[\partial_{\xi_n}\pi^+_{\xi_n}\sigma_{0}(\nabla_{X}^{A}\nabla_{Y}^{A}D_{A}^{-2})\times
\sigma_{-3}(D_{A}^{-2})](x_0)|_{|\xi'|=1}\nonumber\\
=&\mathrm{tr}\bigg[\frac{-ih'(0)}{4(\xi_n-i)^4(\xi_n+i)^2}\bigg(\sum_{k,j=1}^{n-1}X_jY_n\xi_k\xi_j[c(e_k)c(e_n)-\widehat{c}(e_k)\widehat{c}(e_n)]+\sum_{k,l=1}^{n-1}X_nY_l\xi_k\xi_l[c(e_k)c(e_n)-\widehat{c}(e_k)\widehat{c}(e_n)]\bigg)\nonumber\\
&+\frac{i}{2(\xi_n-i)^4(\xi_n+i)^2}\bigg(\sum_{k,j=1}^{n-1}X_jY_n\xi_k\xi_j[c(\partial_k)A+Ac(\partial_k)]+\sum_{k,l=1}^{n-1}X_nY_l\xi_k\xi_l[c(\partial_k)A+Ac(\partial_k)]\bigg)\nonumber\\
&-\frac{h'(0)(3\xi_n^3+7\xi_n)} {4(\xi_n-i)^5(\xi_n+i)^3}\sum_{j,l=1}^{n-1}X_jY_l\xi_j\xi_l+\frac{ih'(0)(3\xi_n^3+7\xi_n)}{4(\xi_n-i)^5(\xi_n+i)^3}X_nY_n\bigg].
\end{align}
By the relation of the Clifford action and ${\rm tr}(AB)={\rm tr }(BA)$, then we have the equalities:
\begin{align}\label{18}
&{\rm tr }[c(e_k)c(e_n)]=0(k<n);~~~~{\rm tr }[\widehat{c}(e_k)\widehat{c}(e_n)]=0(k<n);\nonumber\\
&{\rm tr }[c(\partial_k)A]=0;~~~~{\rm tr }[Ac(\partial_k)]=0.
\end{align}
Therefore, we get
\begin{align}\label{19}
\Phi_4=&i\int_{|\xi'|=1}\int^{+\infty}_{-\infty}
\bigg(-\frac{4h'(0)(3\xi_n^3+7\xi_n)} {(\xi_n-i)^5(\xi_n+i)^3}\sum_{j,l=1}^{n-1}X_jY_l\xi_j\xi_l+\frac{4ih'(0)(3\xi_n^3+7\xi_n)}{(\xi_n-i)^5(\xi_n+i)^3}X_nY_n\bigg)d\xi_n\sigma(\xi')dx'\nonumber\\
=&-4i\sum_{j,l=1}^{n-1}X_jY_lh'(0)\int_{|\xi'|=1}\int_{\Gamma^{+}}\frac{3\xi_n^3+7\xi_n}{(\xi_n-i)^5(\xi_n+i)^2}\xi_j\xi_ld\xi_{n}\sigma(\xi')dx'\nonumber\\
&-4X_nY_nh'(0)\Omega_3\int_{\Gamma^{+}}\frac{3\xi_n^3+7\xi_n}{(\xi_n-i)^5(\xi_n+i)^3}d\xi_{n}dx'\nonumber\\
=&-4i\sum_{j,l=1}^{n-1}X_jY_lh'(0)\frac{4\pi}{3}\frac{2\pi i}{4!}\left[\frac{3\xi_n^3+7\xi_n}{(\xi_n+i)^3}\right]^{(4)}\bigg|_{\xi_n=i}dx'-4X_nY_nh'(0)\Omega_3\frac{2\pi i}{4!}\left[\frac{3\xi_n^3+7\xi_n}{(\xi_n+i)^3}\right]^{(4)}\bigg|_{\xi_n=i}dx'\nonumber\\
=&\frac{17\pi^2}{4}\sum_{j=1}^{n-1}X_jY_jh'(0)dx'-\frac{51i}{16}X_nY_nh'(0)\pi\Omega_3dx'.
\end{align}

\noindent{\bf  Case (c)}~$r=-1,~\ell=-2,~k=j=|\alpha|=0$.\\

\noindent By (\ref{c4}), we get
\begin{align}\label{20}
\Phi_5=-i\int_{|\xi'|=1}\int^{+\infty}_{-\infty}\mathrm{tr} [\pi^+_{\xi_n}
\sigma_{-1}(\nabla_{X}^{A}\nabla_{Y}^{A}D_{A}^{-2})\times
\partial_{\xi_n}\sigma_{-2}(D_{A}^{-2})](x_0)d\xi_n\sigma(\xi')dx'.
\end{align}
By Lemma 4.4, we have
\begin{align}\label{21}
\partial_{\xi_n}\sigma_{-2}(D_{A}^{-2})(x_0)|_{|\xi'|=1}=-\frac{2\xi_n}{(1+\xi_n^2)^2}.
\end{align}
Since
\begin{align}\label{a21}
\sigma_{-1}(\nabla_{X}^{A}\nabla_{Y}^{A}D_{A}^{-2})(x_0)|_{|\xi'|=1}
=&
\sigma_{2}(\nabla_{X}^{A}\nabla_{Y}^{A})\sigma_{-3}(D_{A}^{-2})
+\sigma_{1}(\nabla_{X}^{A}\nabla_{Y}^{A})\sigma_{-2}(D_{A}^{-2})\nonumber\\
&+\sum_{j=1}^{n}\partial_{\xi_{j}}\big[\sigma_{2}(\nabla_{X}^{A}\nabla_{Y}^{A})\big]
D_{x_{j}}\big[\sigma_{-2}(D_{A}^{-2})\big].
\end{align}

(1) Explicit representation the first item of (\ref{a21}),
\begin{align}
&\sigma_{2}(\nabla_{X}^{A}\nabla_{Y}^{A})\sigma_{-3}(D_{A}^{-2})(x_0)|_{|\xi'|=1}\nonumber\\
=&-\sum_{j,l=1}^{n}X_jY_l\xi_j\xi_l\times\Big(-\sqrt{-1}|\xi|^{-4}\xi_k(\Gamma^k-2\sigma^k-2a^k-c(\partial_i)A-Ac(\partial_i))-\sqrt{-1}|\xi|^{-6}2\xi^j\xi_\alpha\xi_\beta\partial_jg^{\alpha\beta}\Big)\nonumber\\
=&-\sum_{j,l=1}^{n}X_jY_l\xi_j\xi_l\times\bigg(-\frac{i}{(1+\xi_n^2)^2}\Big[\frac{1}{2}h'(0)\sum_{k<n}\xi_k[c(e_k)c(e_n)-\widehat{c}(e_k)\widehat{c}(e_n)]-\sum_{k<n}\xi_k[c(\partial_k)A+Ac(\partial_k)]\nonumber\\
&+\frac{3}{2}h'(0)\xi_n\Big]-\frac{2ih'(0)\xi_n}{(1+\xi_n^2)^3}\bigg).
\end{align}

(2) Explicit representation the second item of (\ref{a21}),
\begin{align}
&\sigma_{1}(\nabla_{X}^{A}\nabla_{Y}^{A})\sigma_{-2}(D_{A}^{-2})(x_0)|_{|\xi'|=1}\nonumber\\
=&\Big(\sum_{j,l=1}^nX_j\frac{\partial_{Y_l}}{\partial_{X_j}}\sqrt{-1}\xi_l
+\sum_jB(Y)X_j\sqrt{-1}\xi_j+\sum_lB(X)Y_l\sqrt{-1}\xi_l\nonumber\\
&-\frac{1}{2}\sum_j [c(Y)A+Ac(Y)] X_j\sqrt{-1}\xi_j-\frac{1}{2}\sum_l [c(X)A+Ac(X)]Y_l\sqrt{-1} \xi_l\Big)\times|\xi|^{-2}.
\end{align}

(3) Explicit representation the third item of (\ref{a21}),
\begin{align}
&\sum_{j=1}^{n}\sum_{\alpha}\frac{1}{\alpha!}\partial^{\alpha}_{\xi}\big[\sigma_{2}(\nabla_{X}^{A}\nabla_{Y}^{A})\big]D_x^\alpha\big[\sigma_{-2}(D_{A}^{-2})\big](x_0)|_{|\xi'|=1}\nonumber\\
=&\sum_{j=1}^{n}\partial_{\xi_{j}}\big[\sigma_{2}(\nabla_{X}^{A}\nabla_{Y}^{A})\big](-\sqrt{-1})\partial_{x_j}[\sigma_{-2}(D_{A}^{-2})\big]\nonumber\\
=&\sum_{j=1}^{n}\partial_{\xi_{j}}\big[-\sum_{j,l=1}^nX_jY_l\xi_j\xi_l\big]
(-\sqrt{-1})\partial_{x_{j}}\big[|\xi|^{-2}\big]\nonumber\\
=&\sum_{j=1}^{n}\sum_{l=1}^{n}\sqrt{-1}(X_{j}Y_l+X_{l}Y_j)\xi_{l}\partial_{x_{j}}(|\xi|^{-2}).
\end{align}
We note that $i<n,~\int_{|\xi'|=1}\{\xi_{i_{1}}\xi_{i_{2}}\cdots\xi_{i_{2q+1}}\}\sigma(\xi')=0$, so we omit some items that have no contribution for computing $\Phi_5$.
An easy calculation gives
\begin{align}\label{22}
&\mathrm{tr}[\pi^+_{\xi_n}\sigma_{2}(\nabla_{X}^{A}\nabla_{Y}^{A})\sigma_{-3}(D_{A}^{-2})\times
\partial_{\xi_n}\sigma_{-2}(D_{A}^{-2})](x_0)|_{|\xi'|=1}\nonumber\\
=&-\frac{12h'(0)\xi_n}{(\xi_n-i)^4(\xi_n+i)^2}\sum_{j,l=1}^{n-1}X_jY_l\xi_j\xi_l+\frac{2ih'(0)(2i\xi_n^2+6\xi_n)}{(\xi_n-i)^5(\xi_n+i)^2}\sum_{j,l=1}^{n-1}X_jY_l\xi_j\xi_l\nonumber\\
&-\frac{12h'(0)\xi_n}{(\xi_n-i)^4(\xi_n+i)^2}X_nY_n-\frac{16ih'(0)\xi_n}{(\xi_n-i)^5(\xi_n+i)^2}X_nY_n.
\end{align}
By the relation of the Clifford action and ${\rm tr}(AB)={\rm tr }(BA)$, then we have the equalities:
\begin{align}\label{18}
&{\rm tr }[B(X)]=0;~~{\rm tr }[B(Y)]=0;~~{\rm tr }[c(X)A]=0;\nonumber\\
&{\rm tr }[Ac(X)]=0;~~{\rm tr }[c(Y)A]=0;~~{\rm tr }[Ac(Y)]=0.
\end{align}
Then
\begin{align}\label{23}
&\mathrm{tr}[\pi^+_{\xi_n}\sigma_{1}(\nabla_{X}^{A}\nabla_{Y}^{A})\sigma_{-2}(D_{A}^{-2})\times
\partial_{\xi_n}\sigma_{-2}(D_{A}^{-2})](x_0)|_{|\xi'|=1}\nonumber\\
=&X_n\frac{\partial{Y_n}}{\partial{x_n}} \frac{-16i\xi_n}{(\xi_n-i)^3(\xi_n+i)^2}.
\end{align}
  Also, straightforward computations yield
\begin{align}\label{24}
&\mathrm{tr}\bigg[\pi^+_{\xi_n}\Big(\sum_{j=1}^{n}\sum_{\alpha}\frac{1}{\alpha!}\partial^{\alpha}_{\xi}
\big[\sigma_{2}(\nabla_{X}^{A}\nabla_{Y}^{A})\big]
D_x^{\alpha}\big[\sigma_{-2}(D_{A}^{-2})\big]\Big)\times
\partial_{\xi_n}\sigma_{-2}(D_{A}^{-2})\bigg](x_0)|_{|\xi'|=1}\nonumber\\
=&\frac{-64ih'(0)\xi_n^2}{(\xi_n-i)^2(\xi_n+i)^2}X_nY_n.
\end{align}
Substituting (\ref{22})-(\ref{24}) into (\ref{20}) yields
\begin{align}\label{25}
\Phi_5=&-\frac{11\pi^2}{3}\sum_{j=1}^{n-1}X_jY_jh'(0)dx'-32X_nY_nh'(0)\pi\Omega_3dx'\nonumber\\
&-X_n\frac{\partial{Y_n}}{\partial{x_n}}\frac{\pi}{2}\Omega_3dx'.
\end{align}
Let $X=X^T+X_n\partial_n,~Y=Y^T+Y_n\partial_n,$ then we have $\sum_{j=1}^{n-1}X_jY_j(x_0)=g(X^T,Y^T)(x_0).$
Now $\Phi$ is the sum of the $\Phi_{(1,2,\cdot\cdot\cdot,5)}$. Therefore, we get
\begin{align}\label{26}
\Phi=&\sum_{i=1}^5\Phi_i=-\frac{542+71i}{16}X_nY_nh'(0)\pi\Omega_3dx'-\frac{19\pi^2}{12}g(X^T,Y^T)h'(0)dx'
-X_n\frac{\partial{Y_n}}{\partial{x_n}}\frac{\pi}{2}\Omega_3dx'.
\end{align}
Then, we obtain the following theorem.
\begin{thm}\label{thmb1}
 Let $M$ be a 4-dimensional compact manifold with boundary and $\nabla^{A}$ be an orthogonal
connection. Then we get the spectral Einstein functional associated to $\nabla_{X}^{A}\nabla_{Y}^{A}D_{A}^{-2}$
and $D_{A}^{-2}$ on compact manifolds with boundary
\begin{align}
\label{b2773}
&\widetilde{{\rm Wres}}[\pi^+(\nabla_{X}^{A}\nabla_{Y}^{A}D_{A}^{-2})\circ\pi^+(D_{A}^{-2})]\nonumber\\
=&\frac{4\pi^2}{3}\int_{M}\Big(Ric(V,W)-\frac{1}{2}sg(V,W)\Big) vol_{g}-2\int_{M}sg(V,W) vol_{g}\nonumber\\
&+\int_{\partial M}\Big(-\frac{542+71i}{16}X_nY_nh'(0)\pi\Omega_3-\frac{19\pi^2}{12}g(X^T,Y^T)h'(0)-X_n\frac{\partial{Y_n}}{\partial{x_n}}\frac{\pi}{2}\Omega_3\Big)vol_{\partial M},
\end{align}
where $s$ is the scalar curvature.
\end{thm}

 \section{The residue for the sub-signature operator $\nabla_{X}^{A}\nabla_{Y}^{A}D_{A}^{-1}$ and $D_{A}^{-3}$ }
\label{section:5}
In this section, we compute the 4-dimension spectral Einstein functional for the sub-signature operator $\nabla_{X}^{A}\nabla_{Y}^{A}D_{A}^{-1}$ and $D_{A}^{-3}$.
 Since $[\sigma_{-4}(\nabla_{X}^{A}\nabla_{Y}^{A}D_{A}^{-1}
  \circ D_{A}^{-3})]|_{M}$ has the same expression as $[\sigma_{-4}(\nabla_{X}^{A}\nabla_{Y}^{A}D_{A}^{-1}
  \circ D_{A}^{-3})]|_{M}$ in the case of manifolds without boundary,
so locally we can use Theorem 2.2 to compute the first term.
\begin{thm}
 Let M be a 4-dimensional compact manifold without boundary and $\nabla^{A}$ be an orthogonal
connection. Then we get the spectral Einstein functional associated to $\nabla_{X}^{A}\nabla_{Y}^{A}D_{A}^{-1}$
and $D_{A}^{-3}$ on compact manifolds without boundary
 \begin{align}
&Wres[\sigma_{-4}(\nabla_{X}^{A}\nabla_{Y}^{A}D_{A}^{-1}
  \circ D_{A}^{-3})]\nonumber\\
=&\frac{4\pi^{2}}{3}\int_{M}\Big(Ric(X,Y)-\frac{1}{2}sg(X,Y)\Big) vol_{g}-2\int_{M}sg(X,Y) vol_{g},
\end{align}
where $s$ is the scalar curvature.
\end{thm}
From Lemma 4.2 and Lemma 4.3, we have
\begin{lem} The following identities hold:
\begin{align}
\sigma_{0}(\nabla_{X}^{A}\nabla_{Y}^{A}D_{A}^{-1})=&
\sigma_{2}(\nabla_{X}^{A}\nabla_{Y}^{A})\sigma_{-2}(D_{A}^{-1})
+\sigma_{1}(\nabla_{X}^{A}\nabla_{Y}^{A})\sigma_{-1}(D_{A}^{-1})\nonumber\\
&+\sum_{j=1}^{n}\partial_{\xi_{j}}\big[\sigma_{2}(\nabla_{X}^{A}\nabla_{Y}^{A})\big]
D_{x_{j}}\big[\sigma_{-1}(D_{A}^{-1})\big];\\
\sigma_{1}(\nabla_{X}^{A}\nabla_{Y}^{A}D_{A}^{-1})=&
-\sqrt{-1}\sum_{j,l=1}^nX_jY_l\xi_j\xi_l|\xi|^{-2}.
\end{align}
\end{lem}

Write
 \begin{eqnarray}
D_x^{\alpha}&=(-i)^{|\alpha|}\partial_x^{\alpha};
~\sigma(D_A^3)=p_3+p_2+p_1+p_0;
~\sigma(D_A^{-3})=\sum^{\infty}_{j=3}q_{-j}.
\end{eqnarray}
By the composition formula of pseudo-differential operators, we have
\begin{align}
1=\sigma(D_A^3\circ D_A^{-3})
&=\sum_{\alpha}\frac{1}{\alpha!}\partial^{\alpha}_{\xi}[\sigma(D_A^3)]
D_x^{\alpha}[\sigma(D_A^{-3})]\nonumber\\
&=(p_3+p_2+p_1+p_0)(q_{-3}+q_{-4}+q_{-5}+\cdots)\nonumber\\
&~~~+\sum_j(\partial_{\xi_j}p_3+\partial_{\xi_j}p_2++\partial_{\xi_j}p_1+\partial_{\xi_j}p_0)
(D_{x_j}q_{-3}+D_{x_j}q_{-4}+D_{x_j}q_{-5}+\cdots)\nonumber\\
&=p_3q_{-3}+(p_3q_{-4}+p_2q_{-3}+\sum_j\partial_{\xi_j}p_3D_{x_j}q_{-3})+\cdots,
\end{align}
so
\begin{equation}
q_{-3}=p_3^{-1};~q_{-4}=-p_3^{-1}[p_2p_3^{-1}+\sum_j\partial_{\xi_j}p_3D_{x_j}(p_-3^{-1})].
\end{equation}
Then, it is easy to check that
\begin{lem} The following identities hold:
\begin{align}
&\sigma_{-2}(D_{A}^{-2})=|\xi|^{-2};\\
&\sigma_{-3}(D_{A}^{-2})=-\sqrt{-1}|\xi|^{-4}\xi_k\Big(\Gamma^k-2\sigma^k-2a^k-c(\partial_i)A-Ac(\partial_i)\Big)\nonumber\\
&~~~~~~~~~~~~~~~-\sqrt{-1}|\xi|^{-6}2\xi^j\xi_\alpha\xi_\beta\partial_jg^{\alpha\beta};\\
&\sigma_{-3}(D_{A}^{-3})=\sqrt{-1}c(\xi)|\xi|^{-4};\\
&\sigma_{-4}(D_{A}^{-3})=\frac{c(\xi)\sigma_2(D_{A}^3)c(\xi)}{|\xi|^8}
+\frac{\sqrt{-1}c(\xi)}{|\xi|^8}\Big(|\xi|^4c(dx_n)\partial_{x_n}[c(\xi')]
\nonumber\\
 &~~~~~~~~~~~~~~~~-2h'(0)c(\mathrm{d}x_n)c(\xi)+2\xi_nc(\xi)\partial{x_n}[c(\xi')]+4\xi_nh'(0)\Big),
\end{align}
where
 \begin{align}
\sigma_2(D_{A}^3)=&\sum_{i,j,l}c(dx_l)\partial_l(g^{i,j})\xi_i\xi_j+c(\xi)(4\sigma^k+4a^k-2\Gamma^k)\xi_{k}-2[c(\xi)Ac(\xi)-|\xi|^{2}A]\nonumber\\
&+\frac{1}{4}|\xi|^2\sum_{s,t,l}\omega_{s,t}(e_l)[c(e_l)\widehat{c}(e_s)\widehat{c}(e_t)-c(e_l)c(e_s)c(e_t)]+|\xi|^{2}A.
\end{align}
\end{lem}
Now we  need to compute $\int_{\partial M} \widetilde{\Phi}$. When $n=4$, then ${\rm tr}_{\wedge^*T^*M\otimes\mathbb{C}}[{\rm \texttt{id}}]=2^n=16$, the sum is taken over $
r+l-k-j-|\alpha|=-3,~~r\leq 0,~~l\leq-2,$ then we have the following five cases:\\

\noindent{\bf Case (a)~(I)}~$r=1,~l=-2,~k=j=0,~|\alpha|=1$.\\

\noindent By (\ref{c5}), we get
\begin{equation}
\label{b24}
\widetilde{\Phi}_1=-\int_{|\xi'|=1}\int^{+\infty}_{-\infty}\sum_{|\alpha|=1}
\mathrm{tr}[\partial^\alpha_{\xi'}\pi^+_{\xi_n}\sigma_{1}(\nabla_{X}^{A}\nabla_{Y}^{A}D_{A}^{-1})\times
 \partial^\alpha_{x'}\partial_{\xi_n}\sigma_{-3}(D_{A}^{-3})](x_0)d\xi_n\sigma(\xi')dx'.
\end{equation}
Similarly, for $i<n$,
\begin{align}
\label{b25}
\partial_{x_i}\sigma_{-3}(D_{A}^{-3})(x_0)=&\partial_{x_i}(\sqrt{-1}c(\xi)|\xi|^{-4})(x_0)\nonumber\\
=&\sqrt{-1}\frac{\partial_{x_i}c(\xi)}{|\xi|^4}(x_0)+\sqrt{-1}\frac{c(\xi)\partial_{x_i}(|\xi|^4)}{|\xi|^8}(x_0)\nonumber\\
=&0,
\end{align}
\noindent so $\widetilde{\Phi}_1=0$.\\

\noindent{\bf Case (a)~(II)}~$r=1,~l=-3,~k=|\alpha|=0,~j=1$.\\

\noindent By (\ref{c5}), we get
\begin{equation}
\label{b26}
\widetilde{\Phi}_2=-\frac{1}{2}\int_{|\xi'|=1}\int^{+\infty}_{-\infty}
\mathrm{tr} [\partial_{x_n}\pi^+_{\xi_n}\sigma_{1}(\nabla_{X}^{A}\nabla_{Y}^{A}D_{A}^{-1})\times
\partial_{\xi_n}^2\sigma_{-3}(D_{A}^{-3})](x_0)d\xi_n\sigma(\xi')dx'.
\end{equation}
By Lemma 5.3, we have
\begin{align}\label{27}
\partial_{\xi_n}^2\sigma_{-3}(D_{A}^{-3})(x_0)|_{|\xi'|=1}&=\partial_{\xi_n}^2(c(\xi)|\xi|^{-4})(x_0)|_{|\xi'|=1}\nonumber\\
&=\frac{4i(5\xi_n^2-1)}{(1+\xi_n^2)^4}c(\xi')+\frac{12i(\xi_n^3-\xi_n)}{(1+\xi_n^2)^4}c(dx_n).
\end{align}
By Lemma 5.2, we obtain
\begin{align}\label{28}
\partial_{x_n}\sigma_{1}(\nabla_{X}^{A}\nabla_{Y}^{A}D_{A}^{-1})(x_0)|_{|\xi'|=1}
&=\partial_{x_n}(-\sqrt{-1}\sum_{j,l=1}^nX_jY_l\xi_j\xi_lc(\xi)|\xi|^{-2})(x_0)|_{|\xi'|=1}\nonumber\\
&=\sum_{j,l=1}^nX_jY_l\xi_j\xi_l \left[ \frac{\partial_{x_n}[c(\xi')]}{1+\xi_n^2}+\frac{h'(0)c(\xi)}{(1+\xi_n^2)^2}\right].
\end{align}
Then, we have
\begin{align}\label{29}
\pi^+_{\xi_n}\partial_{x_n}\sigma_{1}(\nabla_{X}^{A}\nabla_{Y}^{A}D_{A}^{-1})(x_0)|_{|\xi'|=1}
=&\partial_{x_n}\pi^+_{\xi_n}\sigma_{1}(\nabla_{X}^{A}\nabla_{Y}^{A}D_{A}^{-1})(x_0)|_{|\xi'|=1}\nonumber\\
=&i\sum_{j,l=1}^{n-1}X_jY_l\xi_j\xi_lh'(0)\left[\frac{ic(\xi')}{4(\xi_n-i)}+\frac{c(\xi')+ic(\mathrm{d}x_n)}{4(\xi_n-i)^2}\right]\nonumber\\
&-\sum_{j,l=1}^{n-1}X_jY_l\xi_j\xi_l\frac{\partial_{x_n}[c(\xi')]}{2(\xi_n-i)}\nonumber\\
&-iX_nY_n\left\{\frac{\partial_{x_n}[c(\xi')]}{2(\xi_n-i)}
 +h'(0)\left[-\frac{2ic(\xi')-3c(\mathrm{d}x_n)}{4(\xi_n-i)}\right.\right.\nonumber\\
&+\left.\left.\frac{[c(\xi')+ic(\mathrm{d}x_n)](i\xi_n+2)}{4(\xi_n-i)^2}\right]\right\}\nonumber\\
&-\sum_{j=1}^{n-1}X_jY_n\xi_j\left[\frac{i\partial_{x_n}[c(\xi')]}{2(\xi_n-i)}
-\frac{ih'(0)[c(\xi')+2ic(\mathrm{d}x_n)]}{4(\xi_n-i)}\right.\nonumber\\
&\left.-\frac{[ic(\xi')-c(\mathrm{d}x_n)](i\xi_n+2)}{(\xi_n-i)^2}\right]\nonumber\\
&-\sum_{l=1}^{n-1}X_nY_l\xi_l \left[\frac{i\partial_{x_n}[c(\xi')]}{2(\xi_n-i)}
-\frac{ih'(0)[c(\xi')+2ic(\mathrm{d}x_n)]}{4(\xi_n-i)}\right.\nonumber\\
&\left.-\frac{[ic(\xi')-c(\mathrm{d}x_n)](i\xi_n+2)}{(\xi_n-i)^2}\right] .
\end{align}
Then, there is the following formula
\begin{align}\label{30}
&\mathrm{tr}[\partial_{x_n}\pi^+_{\xi_n}\sigma_{1}(\nabla_{X}^{A}\nabla_{Y}^{A}D_{A}^{-1})\times
\partial_{\xi_n}^2\sigma_{-3}(D_{A}^{-3})](x_0)|_{|\xi'|=1}\nonumber\\
=&16\sum_{j,l=1}^{n-1}X_jY_l\xi_j\xi_lh'(0)\left(2i\frac{5\xi_n^2-1}{(\xi_n-i)^5(\xi_n+i)^4}
 +\frac{(5\xi_n^2-1)+3i(\xi_n^3-\xi_n)}{(\xi_n-i)^6(\xi_n+i)^4}\right)\nonumber\\
&+16X_nY_nh'(0)\left(\frac{(i-1)(\xi_n^2-1)+12(\xi_n^3-\xi_n)}{(\xi_n-i)^5(\xi_n+i)^4}
 -\frac{(5\xi_n^2-1)+3i(\xi_n^3-\xi_n)}{(\xi_n-i)^6(\xi_n+i)^4}\right)\nonumber\\
&+8\sum_{j=1}^{n-1}X_jY_n\xi_j\left(\frac{[6-3ih'(0)](\xi_n^3-\xi_n)-2i(5\xi_n^2-1)}{(\xi_n-i)^5(\xi_n+i)^4}
 +\frac{2(5\xi_n^2-1)+6i(\xi_n^3-\xi_n)}{(\xi_n-i)^6(\xi_n+i)^4}\right)\nonumber\\
&+8\sum_{l=1}^{n-1}X_nY_l\xi_l\left(\frac{[6-3ih'(0)](\xi_n^3-\xi_n)-2i(5\xi_n^2-1)}{(\xi_n-i)^5(\xi_n+i)^4}
 +\frac{2(5\xi_n^2-1)+6i(\xi_n^3-\xi_n)}{(\xi_n-i)^6(\xi_n+i)^4}\right).
\end{align}
We note that $i<n,~\int_{|\xi'|=1}\{\xi_{i_{1}}\xi_{i_{2}}\cdots\xi_{i_{2q+1}}\}\sigma(\xi')=0$,
so we omit some items that have no contribution for computing $\widetilde{\Phi}_2$.
Therefore, we get
\begin{align}\label{31}
\widetilde{\Phi}_2=&\frac{1}{2}\int_{|\xi'|=1}\int^{+\infty}_{-\infty}\bigg\{
16\sum_{j,l=1}^{n-1}X_jY_l\xi_j\xi_lh'(0)\bigg(2i\frac{5\xi_n^2-1}{(\xi_n-i)^5(\xi_n+i)^4}
 +\frac{(5\xi_n^2-1)+3i(\xi_n^3-\xi_n)}{(\xi_n-i)^6(\xi_n+i)^4}\bigg)\nonumber\\
&+16X_nY_nh'(0)\bigg(\frac{(i-1)(\xi_n^2-1)+12(\xi_n^3-\xi_n)}{(\xi_n-i)^5(\xi_n+i)^4}
 -\frac{(5\xi_n^2-1)+3i(\xi_n^3-\xi_n)}{(\xi_n-i)^6(\xi_n+i)^4}\bigg)
 \bigg\}d\xi_n\sigma(\xi')dx'\nonumber\\
 =&8\sum_{j,l=1}^{n-1}X_jY_lh'(0)\frac{4\pi}{3}\int_{\Gamma^{+}}\bigg(2i\frac{5\xi_n^2-1}{(\xi_n-i)^5(\xi_n+i)^4}
 +\frac{(5\xi_n^2-1)+3i(\xi_n^3-\xi_n)}{(\xi_n-i)^6(\xi_n+i)^4}\bigg)d\xi_{n}dx'\nonumber\\
 &+8X_nY_nh'(0)\Omega_3\int_{\Gamma^{+}}\bigg(\frac{(i-1)(\xi_n^2-1)+12(\xi_n^3-\xi_n)}{(\xi_n-i)^5(\xi_n+i)^4}
 -\frac{(5\xi_n^2-1)+3i(\xi_n^3-\xi_n)}{(\xi_n-i)^6(\xi_n+i)^4}\bigg)d\xi_{n}dx'\nonumber\\
=&-\frac{2368\pi^2}{3}\sum_{j=1}^{n-1}X_jY_jh'(0)dx'-(461+23i)X_nY_nh'(0)\pi\Omega_3dx',
\end{align}
where ${\rm \Omega_{3}}$ is the canonical volume of $S^{2}.$\\

\noindent{\bf Case (a)~(III)}~$r=1,~l=-3,~j=|\alpha|=0,~k=1$.\\

\noindent By (\ref{c5}), we get
\begin{align}\label{32}
\widetilde{\Phi}_3&=-\frac{1}{2}\int_{|\xi'|=1}\int^{+\infty}_{-\infty}
\mathrm{tr} [\partial_{\xi_n}\pi^+_{\xi_n}\sigma_{1}(\nabla_{X}^{A}\nabla_{Y}^{A}D_{A}^{-1})\times
\partial_{\xi_n}\partial_{x_n}\sigma_{-3}(D_{A}^{-3})](x_0)d\xi_n\sigma(\xi')dx'\nonumber\\
&=\frac{1}{2}\int_{|\xi'|=1}\int^{+\infty}_{-\infty}
\mathrm{tr} [\partial_{\xi_n}^2\pi^+_{\xi_n}\sigma_{1}(\nabla_{X}^{A}\nabla_{Y}^{A}D_{A}^{-1})\times
\partial_{x_n}\sigma_{-3}(D_{A}^{-3})](x_0)d\xi_n\sigma(\xi')dx'.
\end{align}
\noindent By Lemma 5.3, we have
\begin{eqnarray}\label{33}
\partial_{x_n}\sigma_{-3}(D_{A}^{-3})(x_0)|_{|\xi'|=1}
=\frac{i\partial_{x_n}[c(\xi')]}{(1+\xi_n^2)^4}-\frac{2ih'(0)c(\xi)}{(1+\xi_n^2)^6}.
\end{eqnarray}
By integrating formula, we obtain
\begin{align}\label{34}
\pi^+_{\xi_n}\sigma_{1}(\nabla_{X}^{A}\nabla_{Y}^{A}D_{A}^{-1})(x_0)|_{|\xi'|=1}
=&-\frac{c(\xi')+ic(\mathrm{d}x_n)}{2(\xi_n-i)}\sum_{j,l=1}^{n-1}X_jY_l\xi_j\xi_l
-\frac{c(\xi')+ic(\mathrm{d}x_n)}{2(\xi_n-i)}X_nY_n\nonumber\\
&-\frac{ic(\xi')-c(\mathrm{d}x_n)}{2(\xi_n-i)}\sum_{j=1}^{n-1}X_jY_n\xi_j
-\frac{ic(\xi')-c(\mathrm{d}x_n)}{2(\xi_n-i)}\sum_{l=1}^{n-1}X_nY_l\xi_l.
\end{align}
Then, we have
\begin{align}\label{35}
\partial_{\xi_n}^2\pi^+_{\xi_n}\sigma_{1}(\nabla_{X}^{A}\nabla_{Y}^{A}D_{A}^{-1})(x_0)|_{|\xi'|=1}
=-\frac{c(\xi')+ic(\mathrm{d}x_n)}{(\xi_n-i)^3}\sum_{j,l=1}^{n-1}X_jY_l\xi_j\xi_l
-\frac{c(\xi')+ic(\mathrm{d}x_n)}{(\xi_n-i)^3}X_nY_n.
\end{align}
We note that $i<n,~\int_{|\xi'|=1}\{\xi_{i_{1}}\xi_{i_{2}}\cdots\xi_{i_{2q+1}}\}\sigma(\xi')=0$,
so we omit some items that have no contribution for computing $\widetilde{\Phi}_3$, then
\begin{align}\label{36}
&\mathrm{tr} [\partial_{\xi_n}\pi^+_{\xi_n}\sigma_{1}(\nabla_{X}^{A}\nabla_{Y}^{A}D_{A}^{-1})\times
\partial_{\xi_n}\partial_{x_n}\sigma_{-3}(D_{A}^{-3})](x_0)|_{|\xi'|=1}\nonumber\\
&=\frac{-8h'(0)}{(\xi_n-i)^5(\xi_n+i)^2}\sum_{j,l=1}^{n-1}X_jY_l\xi_j\xi_l
+\frac{-8h'(0)}{(\xi_n-i)^5(\xi_n+i)^2}X_nY_n.\nonumber\\
\end{align}
Therefore, we get
\begin{align}\label{37}
\widetilde{\Phi}_3=&\frac{1}{2}\int_{|\xi'|=1}\int^{+\infty}_{-\infty}
\bigg(\frac{-8h'(0)}{(\xi_n-i)^5(\xi_n+i)^2}\sum_{j,l=1}^{n-1}X_jY_l\xi_j\xi_l
+\frac{-8h'(0)}{(\xi_n-i)^5(\xi_n+i)^2}X_nY_n\bigg)d\xi_n\sigma(\xi')dx'\nonumber\\
=&-4\sum_{j,l=1}^{n-1}X_jY_lh'(0)\int_{|\xi'|=1}\int_{\Gamma^{+}}\frac{1}{(\xi_n-i)^5(\xi_n+i)^2}\xi_j\xi_ld\xi_{n}\sigma(\xi')dx'\nonumber\\
&-4X_nY_nh'(0)\Omega_3\int_{\Gamma^{+}}\frac{1}{(\xi_n-i)^5(\xi_n+i)^2}d\xi_{n}dx'\nonumber\\
=&-4\sum_{j,l=1}^{n-1}X_jY_lh'(0)\frac{4\pi}{3}\frac{2\pi i}{4!}
\left[\frac{1}{(\xi_n+i)^2}\right]^{(4)}\bigg|_{\xi_n=i}dx'
-4X_nY_nh'(0)\Omega_3\frac{2\pi i}{4!}
\left[\frac{1}{(\xi_n+i)^2}\right]^{(4)}\bigg|_{\xi_n=i}dx'\nonumber\\
=&\frac{10i\pi^2}{3}\sum_{j=1}^{n-1}X_jY_jh'(0)dx'+\frac{5i}{2}X_nY_nh'(0)\pi\Omega_3dx'.
\end{align}

\noindent{\bf Case (b)}~$r=0,~l=-3,~k=j=|\alpha|=0$.\\

\noindent By (\ref{c5}), we get
\begin{align}\label{38}
\widetilde{\Phi}_4&=-i\int_{|\xi'|=1}\int^{+\infty}_{-\infty}
\mathrm{tr}[\pi^+_{\xi_n}\sigma_{0}(\nabla_{X}^{A}\nabla_{Y}^{A}D_{A}^{-1})\times
\partial_{\xi_n}\sigma_{-3}(D_{A}^{-3})](x_0)d\xi_n\sigma(\xi')dx'.
\end{align}
By Lemma 5.3, we obtain
\begin{align}\label{39}
\partial_{\xi_n}\sigma_{-3}(D_{A}^{-3})(x_0)|_{|\xi'|=1}
=\frac{ic(\mathrm{d}x_n)}{(1+\xi_n^2)^2}-\frac{4i\xi_nc(\xi)}{(1+\xi_n^2)^3}.
\end{align}
By Lemma 5.2, we have
\begin{align}\label{c6}
\sigma_{0}(\nabla_{X}^{A}\nabla_{Y}^{A}D_{A}^{-1})=&
\sigma_{2}(\nabla_{X}^{A}\nabla_{Y}^{A})\sigma_{-2}(D_{A}^{-1})
+\sigma_{1}(\nabla_{X}^{A}\nabla_{Y}^{A})\sigma_{-1}(D_{A}^{-1})\nonumber\\
&+\sum_{j=1}^{n}\partial _{\xi_{j}}\big[\sigma_{2}(\nabla_{X}^{A}\nabla_{Y}^{A})\big]
D_{x_{j}}\big[\sigma_{-1}(D_{A}^{-1})\big].
\end{align}

(1) Explicit representation the first item of (\ref{c6}),
\begin{align}
&\sigma_{2}(\nabla_{X}^{A}\nabla_{Y}^{A})\sigma_{-2}(D_{A}^{-1})(x_0)|_{|\xi'|=1}\nonumber\\
=&-\sum_{j,l=1}^{n}X_jY_l\xi_j\xi_l\bigg[\frac{c(\xi)\sigma_0(D_{A})c(\xi)}{|\xi|^4}
+\frac{c(\xi)}{|\xi|^6}\sum_jc(\mathrm{d}x_j)\Big[\partial_{x_j}[c(\xi)]|\xi|^2-c(\xi)\partial_{x_j}(|\xi|^2)\Big]\bigg],
\end{align}
where $\sigma_{0}(D_{A})=\frac{1}{4}\sum_{s,t,i}\omega_{s,t}(e_i)c(e_i)\widehat{c}(e_s)\widehat{c}(e_t)-\frac{1}{4}\sum_{s,t,i}\omega_{s,t}(e_i)c(e_i)c(e_s)c(e_t)+A.$\\

\noindent We denote
\begin{align}\label{40}
P_1(x_0)&=\frac{1}{4}\sum_{s,t,i}\omega_{s,t}(e_i)
(x_{0})c(e_i)\widehat{c}(e_s)\widehat{c}(e_t);\nonumber\\
P_2(x_0)&=-\frac{1}{4}\sum_{s,t,i}\omega_{s,t}(e_i)
(x_{0})c(e_i)c(e_s)c(e_t).
\end{align}
Then
\begin{align}\label{41}
\pi^+_{\xi_n}\sigma_{-2}({D_A}^{-1})(x_0)|_{|\xi'|=1}=&\pi^+_{\xi_n}\Big[\frac{c(\xi)P_1(x_0)c(\xi)}{(1+\xi_n^2)^2}\Big]+\pi^+_{\xi_n}
\Big[\frac{c(\xi)A(x_0)c(\xi)}{(1+\xi_n^2)^2}\Big]
\nonumber\\
&+\pi^+_{\xi_n}\Big[\frac{c(\xi)P_2(x_0)c(\xi)+c(\xi)c(dx_n)\partial_{x_n}[c(\xi')](x_0)}{(1+\xi_n^2)^2}-h'(0)\frac{c(\xi)c(dx_n)c(\xi)}{(1+\xi_n^{2})^3}\Big].
\end{align}
By computations, we have
\begin{align}\label{42}
\pi^+_{\xi_n}\Big[\frac{c(\xi)P_1(x_0)c(\xi)}{(1+\xi_n^2)^2}\Big]=&\pi^+_{\xi_n}\Big[\frac{c(\xi')P_1(x_0)c(\xi')}{(1+\xi_n^2)^2}\Big]
+\pi^+_{\xi_n}\Big[ \frac{\xi_nc(\xi')P_1(x_0)c(dx_{n})}{(1+\xi_n^2)^2}\Big]\nonumber\\
&+\pi^+_{\xi_n}\Big[\frac{\xi_nc(dx_{n})P_1(x_0)c(\xi')}{(1+\xi_n^2)^2}\Big]
+\pi^+_{\xi_n}\Big[\frac{\xi_n^{2}c(dx_{n})P_1(x_0)c(dx_{n})}{(1+\xi_n^2)^2}\Big]\nonumber\\
=&\frac{-c(\xi')P_1(x_0)c(\xi')(2+i\xi_{n})}{4(\xi_{n}-i)^{2}}
+\frac{-ic(\xi')P_1(x_0)c(dx_{n})}{4(\xi_{n}-i)^{2}}\nonumber\\
&+\frac{-ic(dx_{n})P_1(x_0)c(\xi')}{4(\xi_{n}-i)^{2}}
+\frac{-i\xi_{n}c(dx_{n})P_1(x_0)c(dx_{n})}{4(\xi_{n}-i)^{2}},
\end{align}
and
\begin{align}\label{43}
\pi^+_{\xi_n}\Big[\xi_n^2\frac{c(\xi)P_1(x_0)c(\xi)}{(1+\xi_n^2)^2}\Big]=&\pi^+_{\xi_n}\Big[\frac{\xi_n^2c(\xi')P_1(x_0)c(\xi')}{(1+\xi_n^2)^2}\Big]
+\pi^+_{\xi_n}\Big[ \frac{\xi_n^3c(\xi')P_1(x_0)c(dx_{n})}{(1+\xi_n^2)^2}\Big]\nonumber\\
&+\pi^+_{\xi_n}\Big[\frac{\xi_n^3c(dx_{n})P_1(x_0)c(\xi')}{(1+\xi_n^2)^2}\Big]
+\pi^+_{\xi_n}\Big[\frac{\xi_n^4c(dx_{n})P_1(x_0)c(dx_{n})}{(1+\xi_n^2)^2}\Big]\nonumber\\
=&\frac{-c(\xi')P_1(x_0)c(\xi')i\xi_{n}}{4(\xi_{n}-i)^{2}}
+\frac{c(\xi')P_1(x_0)c(dx_{n})(2\xi_n-i)}{4(\xi_{n}-i)^{2}}\nonumber\\
&+\frac{c(dx_{n})P_1(x_0)c(\xi')(2\xi_n-i)}{4(\xi_{n}-i)^{2}}
+\frac{c(dx_{n})P_1(x_0)c(dx_{n})(3i\xi_n+2)}{4(\xi_{n}-i)^{2}}.
\end{align}
Since
\begin{align}\label{44}
c(dx_n)P_1(x_0)
&=-\frac{1}{4}h'(0)\sum^{n-1}_{i=1}c(e_i)
\widehat{c}(e_i)c(e_n)\widehat{c}(e_n),
\end{align}
then by the relation of the Clifford action and ${\rm tr}{(AB)}={\rm tr }{(BA)}$,  we have the equalities:
\begin{align}\label{45}
&{\rm tr}[c(e_i)
\widehat{c}(e_i)c(e_n)
\widehat{c}(e_n)]=0~~(i<n);~~
{\rm tr}[c(\xi')c(dx_n)]=0;\nonumber\\
&{\rm tr}[P_1(x_0)c(dx_n)]=0;~~
{\rm tr}[P_2(x_0)c(dx_n)]=12h'(0);\nonumber\\
&{\rm tr}[\partial_{x_n}[c(\xi')]c(dx_n)]=0;~~{\rm tr}[\partial_{x_n}[c(\xi')]c(\xi')](x_0)|_{|\xi'|=1}=-8h'(0).
\end{align}
By (\ref{39}), (\ref{42}) and (\ref{43}), we have
\begin{align}\label{46}
&{\rm tr }\bigg[\pi^+_{\xi_n}\Big[-\sum_{j,l=1}^{n}X_jY_l\xi_j\xi_l\frac{c(\xi)P_1(x_0)c(\xi)}{(1+\xi_n^2)^2}\Big]
\times\partial_{\xi_n}\sigma_{-3}(D_A^{-3})\bigg](x_0)|_{|\xi'|=1}\nonumber\\
=&-\sum_{j,l=1}^{n-1}X_jY_l\xi_j\xi_l{\rm tr }\bigg[\pi^+_{\xi_n}\Big[\frac{c(\xi)P_1(x_0)c(\xi)}{(1+\xi_n^2)^2}\Big]
\times\partial_{\xi_n}\sigma_{-3}(D_A^{-3})\bigg](x_0)|_{|\xi'|=1}\nonumber\\
&-X_nY_n{\rm tr }\bigg[\pi^+_{\xi_n}\Big[\xi_n^2\frac{c(\xi)P_1(x_0)c(\xi)}{(1+\xi_n^2)^2}\Big]
\times\partial_{\xi_n}\sigma_{-3}(D_A^{-3})\bigg](x_0)|_{|\xi'|=1}\nonumber\\
=&-\sum_{j,l=1}^{n-1}X_jY_l\xi_j\xi_l\bigg(\frac{-1}{2(\xi_n-i)^4(\xi_n+i)^2}{\rm tr}[c(\xi')P_1(x_0)]+\frac{-2i\xi_n+2\xi_n}{(\xi_n-i)^5(\xi_n+i)^3}{\rm tr}[c(\xi')P_1(x_0)]\bigg)\nonumber\\
&-X_nY_n\bigg(\frac{-2i\xi_n-1}{2(\xi_n-i)^4(\xi_n+i)^2}{\rm tr}[c(\xi')P_1(x_0)]+\frac{(-2+4i)\xi_n+(2+2i)\xi_n}{(\xi_n-i)^5(\xi_n+i)^3}{\rm tr}[c(\xi')P_1(x_0)]\bigg).
\end{align}

\noindent We note that $i<n,~\int_{|\xi'|=1}\{\xi_{i_{1}}\xi_{i_{2}}\cdots\xi_{i_{2q+1}}\}\sigma(\xi')=0$,
so ${\rm tr }[c(\xi')P_1(x_0)]$ has no contribution for computing $\widetilde{\Phi}_4$.\\
Similar to (\ref{46}), we have
\begin{align}
&{\rm tr }\bigg[\pi^+_{\xi_n}\Big[-\sum_{j,l=1}^{n}X_jY_l\xi_j\xi_l\frac{c(\xi)A(x_0)c(\xi)}{(1+\xi_n^2)^2}\Big]
\times\partial_{\xi_n}\sigma_{-3}(D_A^{-3})\bigg](x_0)|_{|\xi'|=1}\nonumber\\
=&-\sum_{j,l=1}^{n-1}X_jY_l\xi_j\xi_l{\rm tr }\bigg[\pi^+_{\xi_n}\Big[\frac{c(\xi)A(x_0)c(\xi)}{(1+\xi_n^2)^2}\Big]
\times\partial_{\xi_n}\sigma_{-3}(D_A^{-3})\bigg](x_0)|_{|\xi'|=1}\nonumber\\
&-X_nY_n{\rm tr }\bigg[\pi^+_{\xi_n}\Big[\xi_n^2\frac{c(\xi)A(x_0)c(\xi)}{(1+\xi_n^2)^2}\Big]
\times\partial_{\xi_n}\sigma_{-3}(D_A^{-3})\bigg](x_0)|_{|\xi'|=1}\nonumber\\
=&0.
\end{align}
By computations, we have
\begin{eqnarray}\label{47}
\pi^+_{\xi_n}\Big[\frac{c(\xi)P_2(x_0)c(\xi)+c(\xi)c(dx_n)\partial_{x_n}[c(\xi')](x_0)}{(1+\xi_n^2)^2}\Big]-h'(0)\pi^+_{\xi_n}\Big[\frac{c(\xi)c(dx_n)c(\xi)}{(1+\xi_n^2)^3}\Big]:= C_1-C_2,
\end{eqnarray}
where
\begin{align}\label{48}
C_1&=\frac{-1}{4(\xi_n-i)^2}\big[(2+i\xi_n)c(\xi')P_2(x_0)c(\xi')+i\xi_nc(dx_n)P_2(x_0)c(dx_n)\nonumber\\
&+(2+i\xi_n)c(\xi')c(dx_n)\partial_{x_n}[c(\xi')](x_0)+ic(dx_n)P_2(x_0)c(\xi')
+ic(\xi')P_2(x_0)c(dx_n)-i\partial_{x_n}[c(\xi')](x_0)\big],
\end{align}
and
\begin{align}\label{49}
C_2&=\frac{h'(0)}{2}\left[\frac{c(dx_n)}{4i(\xi_n-i)}+\frac{c(dx_n)-ic(\xi')}{8(\xi_n-i)^2}
+\frac{3\xi_n-7i}{8(\xi_n-i)^3}[ic(\xi')-c(dx_n)]\right],
\end{align}
where $P_2(x_0)=c_0c(dx_n)$ and $c_0=-\frac{3}{4}h'(0)$.\\

\noindent By (\ref{39}), (\ref{48}) and (\ref{49}), we have
\begin{align}\label{50}
&{\rm tr }\bigg[\Big(-\sum_{j,l=1}^{n}X_jY_l\xi_j\xi_l(C_1-C_2)\Big)
\times\partial_{\xi_n}\sigma_{-3}(D_A^{-3})\bigg](x_0)|_{|\xi'|=1}\nonumber\\
=&-\sum_{j,l=1}^{n-1}X_jY_l\xi_j\xi_l{\rm tr }\big[(C_1-C_2)
\times\partial_{\xi_n}\sigma_{-3}(D_A^{-3})\big](x_0)|_{|\xi'|=1}\nonumber\\
&-X_nY_n{\rm tr }\bigg[\pi^+_{\xi_n}\Big(\xi_n^2\frac{c(\xi)P_2(x_0)c(\xi)+c(\xi)c(dx_n)\partial_{x_n}[c(\xi')](x_0)}{(1+\xi_n^2)^2}\nonumber\\
&-h'(0)\xi_n^2\frac{c(\xi)c(dx_n)c(\xi)}{(1+\xi_n^{2})^3}\Big)\times\partial_{\xi_n}\sigma_{-3}(D_A^{-3})\bigg](x_0)|_{|\xi'|=1}\nonumber\\
=&-\sum_{j,l=1}^{n-1}X_jY_l\xi_j\xi_l\bigg(2h'(0)\frac{\xi_n^2-6i\xi_n-7}{(\xi_n-i)^5(\xi_n+i)^2}+2ih'(0)\frac{12\xi_n^3-36i\xi_n^2-23\xi_n}{(\xi_n-i)^6(\xi_n+i)^3}\bigg)\nonumber\\
&-X_nY_n\bigg(6h'(0)\frac{2\xi_n-i}{(\xi_n-i)^4(\xi_n+i)^2}+16h'(0)\frac{\xi_n^4+i\xi_n^3+2i\xi_n}{(\xi_n-i)^6(\xi_n+i)^3}\bigg).
\end{align}
\noindent Substituting (\ref{50}) into (\ref{38}) yields
\begin{align}\label{51}
&-i\int_{|\xi'|=1}\int^{+\infty}_{-\infty}
\mathrm{tr} \bigg[\pi^+_{\xi_n}\Big(\sigma_{2}(\nabla_{X}^{A}\nabla_{Y}^{A})\sigma_{-2}(D_{A}^{-1})\Big)\times
\partial_{\xi_n}\sigma_{-3}(D_{A}^{-3})\bigg](x_0)d\xi_n\sigma(\xi')dx'\nonumber\\
=&-i\int_{|\xi'|=1}\int^{+\infty}_{-\infty}\bigg\{-\sum_{j,l=1}^{n-1}X_jY_l\xi_j\xi_lh'(0)\bigg(2\frac{\xi_n^2-6i\xi_n-7}{(\xi_n-i)^5(\xi_n+i)^2}+2i\frac{12\xi_n^3-36i\xi_n^2-23\xi_n}{(\xi_n-i)^6(\xi_n+i)^3}\bigg)\nonumber\\
&-X_nY_nh'(0)\bigg(6\frac{2\xi_n-i}{(\xi_n-i)^4(\xi_n+i)^2}+16\frac{\xi_n^4+i\xi_n^3+2i\xi_n}{(\xi_n-i)^6(\xi_n+i)^3}\bigg)\bigg\}(x_0)d\xi_n\sigma(\xi')dx'\nonumber\\
=&\sum_{j,l=1}^{n-1}X_jY_lh'(0)\frac{4\pi}{3}\int_{\Gamma^{+}}\bigg(2i\frac{\xi_n^2-6i\xi_n-7}{(\xi_n-i)^5(\xi_n+i)^2}-2\frac{12\xi_n^3-36i\xi_n^2-23\xi_n}{(\xi_n-i)^6(\xi_n+i)^3}\bigg)d\xi_ndx'\nonumber\\
&+X_nY_nh'(0)\Omega_3\int_{\Gamma^{+}}\bigg(6i\frac{2\xi_n-i}{(\xi_n-i)^4(\xi_n+i)^2}+16i\frac{\xi_n^4+i\xi_n^3+2i\xi_n}{(\xi_n-i)^6(\xi_n+i)^3}\bigg)\bigg\}d\xi_ndx'\nonumber\\
=&\sum_{j,l=1}^{n-1}X_jY_lh'(0)\frac{4\pi}{3}\bigg(2i\frac{2\pi i}{4!}\left[\frac{\xi_n^2-6i\xi_n-7}{(\xi_n+i)^2}\right]^{(4)}\bigg|_{\xi_n=i}-2\frac{2\pi i}{5!}\left[\frac{12\xi_n^3-36i\xi_n^2-23\xi_n}{(\xi_n+i)^3}\right]^{(5)}\bigg|_{\xi_n=i}\bigg)dx'\nonumber\\
&+X_nY_nh'(0)\Omega_3\bigg(6i\frac{2\pi i}{3!}\left[\frac{2\xi_n-i}{(\xi_n+i)^2}\right]^{(3)}\bigg|_{\xi_n=i}+16i\frac{2\pi i}{5!}\left[\frac{\xi_n^4+i\xi_n^3+2i\xi_n}{(\xi_n+i)^3}\right]^{(5)}\bigg|_{\xi_n=i}\bigg)dx'\nonumber\\
=&\frac{55\pi^2}{3}\sum_{j,l=1}^{n-1}X_jY_lh'(0)dx'-\frac{3}{8}X_nY_nh'(0)\pi\Omega_3dx'.
\end{align}

(2) Explicit representation the second item of (\ref{c6}),
\begin{align}
&\sigma_{1}(\nabla_{X}^{A}\nabla_{Y}^{A})\sigma_{-1}(D_{A}^{-1})(x_0)|_{|\xi'|=1}\nonumber\\
=&\Big(\sum_{j,l=1}^nX_j\frac{\partial_{Y_l}}{\partial_{X_j}}\sqrt{-1}\xi_l
+\sum_jB(Y)X_j\sqrt{-1}\xi_j+\sum_lB(X)Y_l\sqrt{-1}\xi_l\nonumber\\
&-\frac{1}{2}\sum_j [c(Y)A+Ac(Y)] X_j\sqrt{-1}\xi_j-\frac{1}{2}\sum_l [c(X)A+Ac(X)]Y_l\sqrt{-1} \xi_l\Big)\frac{\sqrt{-1}c(\xi)}{|\xi|^{2}}.
\end{align}
By integrating formula, we get
\begin{align}
&\pi^+_{\xi_n}\bigg[\Big(-\frac{1}{2}\sum_j [c(Y)A+Ac(Y)] X_j\sqrt{-1}\xi_j-\frac{1}{2}\sum_l [c(X)A+Ac(X)]Y_l\sqrt{-1} \xi_l\Big)\frac{\sqrt{-1}c(\xi)}{|\xi|^{2}}\bigg]\nonumber\\
=&\pi^+_{\xi_n}\bigg[\Big(-\frac{1}{2}\sum_{j=1}^{n-1} [c(Y)A+Ac(Y)] X_j\sqrt{-1}\xi_j-\frac{1}{2}\sum_{l=1}^{n-1} [c(X)A+Ac(X)]Y_l\sqrt{-1} \xi_l\Big)\frac{\sqrt{-1}c(\xi)}{|\xi|^{2}}\bigg]\nonumber\\
&+\pi^+_{\xi_n}\bigg[\Big(-\frac{1}{2}[c(Y)A+Ac(Y)] X_n\sqrt{-1}\xi_n-\frac{1}{2}[c(X)A+Ac(X)]Y_n\sqrt{-1} \xi_n\Big)\frac{\sqrt{-1}c(\xi)}{|\xi|^{2}}\bigg]\nonumber\\
=&\frac{1}{2}\Big(\sum_{j=1}^{n-1} [c(Y)A+Ac(Y)] X_j\xi_j+\sum_{l=1}^{n-1} [c(X)A+Ac(X)]Y_l \xi_l\Big)\frac{-ic(\xi')+c(dx_n)}{2(\xi_{n}-i)}\nonumber\\
&+\frac{1}{2}\Big([c(Y)A+Ac(Y)]X_n+[c(X)A+Ac(X)]Y_n\Big)\frac{c(\xi')+ic(dx_n)}{2(\xi_{n}-i)}.
\end{align}
We note that $i<n,~\int_{|\xi'|=1}\{\xi_{i_{1}}\xi_{i_{2}}\cdots\xi_{i_{2q+1}}\}\sigma(\xi')=0$,
and by the relation of the Clifford action and ${\rm tr}(AB)={\rm tr }(BA)$, then we have the equalities:
\begin{align}\label{52}
&{\rm tr }[c(X)Ac(\xi')c(dx_n)]=0;~~~~{\rm tr }[Ac(X)c(\xi')c(dx_n)]=0;\nonumber\\
&{\rm tr }[c(X)Ac(dx_n)c(\xi')]=0;~~~~{\rm tr }[Ac(X)c(dx_n)c(\xi')]=0,
\end{align}
so
\begin{align}\label{53}
\mathrm{tr} \Big[\pi^+_{\xi_n}\Big(\sigma_{1}(\nabla_{X}^{A}\nabla_{Y}^{A})\sigma_{-1}(D_{A}^{-1})\Big)\times
\partial_{\xi_n}\sigma_{-3}(D_{A}^{-1})\Big](x_0)|_{|\xi'|=1}=0.
\end{align}

(3) Explicit representation the third item of (\ref{c6}),
\begin{align}
\sum_{j=1}^{n}\sum_{\alpha}\frac{1}{\alpha!}\partial^{\alpha}_{\xi}\big[\sigma_{2}(\nabla_{X}^{A}\nabla_{Y}^{A})\big]
D_x^{\alpha}\big[\sigma_{-1}(D_{A}^{-1})\big](x_0)|_{|\xi'|=1}
=&\sum_{j=1}^{n}\partial_{\xi_{j}}
\big[\sigma_{2}(\nabla_{X}^{A}\nabla_{Y}^{A})\big]
(-\sqrt{-1})\partial_{x_{j}}\big[\sigma_{-1}( D_{A} ^{-1})\big]\nonumber\\
=&\sum_{j=1}^{n}\partial_{\xi_{j}}\Big[-\sum_{j,l=1}^nX_jY_l\xi_j\xi_l\Big]
(-\sqrt{-1})\partial_{x_{j}}\Big(\frac{\sqrt{-1}c(\xi)}{|\xi|^{2}}\Big)\nonumber\\
=&\sum_{j=1}^{n}\sum_{l=1}^{n}\sqrt{-1}(X_{j}Y_l+X_{l}Y_j)\xi_{l}\partial_{x_{j}}\Big(\frac{\sqrt{-1}c(\xi)}{|\xi|^{2}}\Big).
\end{align}
By integrating formula, we obtain
\begin{align}\label{c7}
&\pi^+_{\xi_n}\bigg[\sum_{j=1}^{n}\sum_{\alpha}\frac{1}{\alpha!}\partial^{\alpha}_{\xi}\big[\sigma_{2}(\nabla_{X}^{A}\nabla_{Y}^{A})\big]
D_x^{\alpha}\big[\sigma_{-1}(D_{A}^{-1})\big]\bigg]\nonumber\\
=&\pi^+_{\xi_n}\bigg[\sum_{l=1}^{n-1}\sqrt{-1}(X_{n}Y_l+X_{l}Y_n)\xi_{l}\partial_{x_{n}}\Big(\frac{\sqrt{-1}c(\xi)}{|\xi|^{2}}\Big)
\bigg]+\pi^+_{\xi_n}\bigg[\sqrt{-1}(X_{n}Y_n+X_{n}Y_n)\xi_{n}\partial_{x_{n}}\Big(\frac{\sqrt{-1}c(\xi)}{|\xi|^{2}}\Big)
\bigg]\nonumber\\
=&\sum_{l=1}^{n-1}(X_{n}Y_l+X_{l}Y_n)\xi_{l}\Big[\frac{i\partial_{x_{n}}[c(\xi')]}{2(\xi_n-i)}
-h'(0)\frac{(2+i\xi_n )c(\xi')}{4(\xi_n-i)^2}
-h'(0)\frac{ic(dx_n)}{4(\xi_n-i)^2}\Big]\nonumber\\
&+X_{n}Y_n\Big[\frac{-\partial_{x_{n}}[c(\xi')]}{(\xi_n-i)}
-h'(0)\frac{ic(\xi')}{2(\xi_n-i)^2}
+h'(0)\frac{i\xi_n c(dx_n)}{2(\xi_n-i)^2}\Big].
\end{align}
Substituting (\ref{c7}) into (\ref{38}) yields
\begin{align}\label{54}
&-i\int_{|\xi'|=1}\int^{+\infty}_{-\infty}
\mathrm{tr}\Big[\pi^+_{\xi_n}\Big(\sum_{j=1}^{n}\sum_{\alpha}\frac{1}{\alpha!}\partial^{\alpha}_{\xi}
\big[\sigma_{2}(\nabla_{X}^{A}\nabla_{Y}^{A})\big]
D_x^{\alpha}\big[\sigma_{-1}(D_{A}^{-1})\big]\Big)\nonumber\\
&\times\partial_{\xi_n}\sigma_{-3}(D_{A}^{-3})\Big](x_0)d\xi_n\sigma(\xi')dx'\nonumber\\
=&-i\int_{|\xi'|=1}\int^{+\infty}_{-\infty}-8X_nY_nh'(0)\frac{-3\xi_n^3+4i\xi_n^2-3\xi_n+4}{(\xi_n-i)^5(\xi_n+i)^3}d\xi_n\sigma(\xi')dx'\nonumber\\
=&8iX_nY_nh'(0)\Omega_3\int_{\Gamma^{+}}\frac{-3\xi_n^3+4i\xi_n^2-3\xi_n+4}{(\xi_n-i)^5(\xi_n+i)^3}d\xi_ndx'\nonumber\\
=&8iX_nY_nh'(0)\Omega_3\frac{2\pi i}{4!}\left[\frac{-3\xi_n^3+4i\xi_n^2-3\xi_n+4}{(\xi_n+i)^3}\right]^{(4)}\bigg|_{\xi_n=i}dx'\nonumber\\
=&\Big(\frac{7}{2}-\frac{15i}{2}\Big)X_{n}Y_nh'(0)\pi\Omega_3dx'.
\end{align}
Summing up (1), (2) and (3) leads to the desired equality
\begin{align}\label{55}
\widetilde{\Phi}_4
&=\frac{55\pi^2}{3}\sum_{j=1}^{n-1}X_jY_jh'(0)dx'+\Big(\frac{25}{8}-\frac{15i}{2}\Big)X_nY_nh'(0)\pi\Omega_3dx'.
\end{align}

\noindent {\bf  Case (c)}~$r=1,~\ell=-4,~k=j=|\alpha|=0$.\\

\noindent By (\ref{c5}), we get
\begin{align}\label{56}
\widetilde{\Phi}_5&=-\int_{|\xi'|=1}\int^{+\infty}_{-\infty}\mathrm{tr} [\pi^+_{\xi_n}
\sigma_{1}(\nabla_{X}^{A}\nabla_{Y}^{A}D_{A}^{-1})\times
\partial_{\xi_n}\sigma_{-4}(D_{A}^{-3})](x_0)d\xi_n\sigma(\xi')dx'\nonumber\\
&=\int_{|\xi'|=1}\int^{+\infty}_{-\infty}\mathrm{tr}
[\partial_{\xi_n}\pi^+_{\xi_n}\sigma_{1}(\nabla_{X}^{A}\nabla_{Y}^{A}D_{A}^{-1})\times
\sigma_{-4}(D_{A}^{-3})](x_0)d\xi_n\sigma(\xi')dx'.
\end{align}
An easy calculation gives
\begin{align}\label{57}
\partial_{\xi_n}\pi^+_{\xi_n}\sigma_{1}(\nabla_{X}^{A}\nabla_{Y}^{A}D_{A}^{-1})(x_0)|_{|\xi'|=1}
=&\frac{c(\xi')+ic(dx_n)}{2(\xi_n-i)^2}\sum_{j,l=1}^{n-1}X_jY_l\xi_j\xi_l
-\frac{c(\xi')+ic(dx_n)}{2(\xi_n-i)^2}X_nY_n \nonumber\\
&+\frac{ic(\xi')-c(dx_n)}{2(\xi_n-i)^2}\sum_{j=1}^{n}X_jY_n\xi_j
+\frac{ic(\xi')-c(dx_n)}{2(\xi_n-i)^2}\sum_{l=1}^{n}X_nY_l\xi_l\nonumber\\
=&\sum_{j,l=1}^{n-1}X_jY_l\xi_j\xi_l\frac{1}{2(\xi_n-i)^2}c(\xi')+\sum_{j,l=1}^{n-1}X_jY_l\xi_j\xi_l\frac{i}{2(\xi_n-i)^2}c(dx_n)\nonumber\\
&+X_nY_n\frac{2i\xi_n-1}{2(\xi_n-i)^2}c(\xi')+X_nY_n\frac{-2\xi_n-i}{2(\xi_n-i)^2}c(dx_n)\nonumber\\
&+\sum_{j=1}^{n-1}(X_jY_n+X_nY_j)\xi_j\frac{i}{(\xi_n-i)^2}c(\xi')\nonumber\\
&+\sum_{j=1}^{n-1}(X_jY_n+X_nY_j)\xi_j\frac{-1}{(\xi_n-i)^2}c(dx_n).
\end{align}
By (4.62) in \cite{WWW}, we have
\begin{align}\label{58}
\sigma_{-4}({D_A}^{-3})(x_{0})|_{|\xi'|=1}=&
\frac{c(\xi)\sigma_{2}({D_A}^{3})
(x_{0})|_{|\xi'|=1}c(\xi)}{|\xi|^8}
-\frac{c(\xi)}{|\xi|^4}\sum_j\partial_{\xi_j}\big(c(\xi)|\xi|^2\big)
D_{x_j}\Big(\frac{\sqrt{-1}c(\xi)}{|\xi|^4}\Big)\nonumber\\
=&\frac{1}{|\xi|^8}c(\xi)\Big(\frac{1}{2}h'(0)c(\xi)\sum_{k<n}\xi_k
c(e_k)c(e_n)-\frac{1}{2}h'(0)c(\xi)\sum_{k<n}\xi_k
\widehat{c}(e_k)\widehat{c}(e_n)\nonumber\\
&-\frac{5}{2}h'(0)\xi_nc(\xi)-\frac{1}{4}h'(0)|\xi|^2c(dx_n)
-2c(\xi)Ac(\xi)+3|\xi|^2A]\Big)c(\xi)\nonumber\\
&+\frac{ic(\xi)}{|\xi|^8}\Big(|\xi|^4c(dx_n)\partial_{x_n}[c(\xi')]
-2h'(0)c(dx_n)c(\xi)+2\xi_{n}c(\xi)\partial_{x_n}[c(\xi')]+4\xi_{n}h'(0)\Big).\nonumber\\
\end{align}

\noindent We note that $i<n,~\int_{|\xi'|=1}\{\xi_{i_{1}}\xi_{i_{2}}\cdots\xi_{i_{2q+1}}\}\sigma(\xi')=0$,
so we omit some items that have no contribution for computing $\widetilde{\Phi}_5$. Here
\begin{align}\label{59}
&{\rm tr}[c(e_i)
\widehat{c}(e_i)c(e_n)
\widehat{c}(e_n)]=0~~(i<n);~~
{\rm tr}[c(\xi')c(dx_n)]=0;\nonumber\\
&{\rm tr}[Ac(\xi')]=0;~~
{\rm tr}[Ac(dx_n)]=0;\nonumber\\
&{\rm tr}[\partial_{x_n}[c(\xi')]c(dx_n)]=0;~~{\rm tr}[\partial_{x_n}[c(\xi')]c(\xi')](x_0)|_{|\xi'|=1}=-8h'(0).
\end{align}
Also, straightforward computations yield
\begin{align}\label{60}
&\mathrm{tr}[\partial_{\xi_n}\pi^+_{\xi_n}\sigma_{1}(\nabla_{X}^{A}\nabla_{Y}^{A}D_{A}^{-1})\times\sigma_{-4}(D_{A}^{-3})](x_0)|_{|\xi'|=1}\nonumber\\
=&\sum_{j,l=1}^{n-1}X_jY_l\xi_j\xi_lh'(0)\bigg[\frac{1}{(\xi_n-i)^2(1+\xi_n^2)^3}[4i\xi_n^3+(4-22i)\xi_n^2-(24-12i+4\xi_i^2)\xi_n+2i\xi_i^2+4]\nonumber\\
&+\frac{1}{(\xi_n-i)^2(1+\xi_n^2)^4}[48\xi_n^2-64i\xi_n-16\xi_i^2]\bigg]\nonumber\\
&+X_nY_n\bigg[\frac{1}{(\xi_n-i)^2(1+\xi_n^2)^3}[-8\xi_n^4+(44+4i)\xi_n^3-(28+26i+8i\xi_i^2)\xi_n^2+(16-4i)\xi_n-6i\xi_i^2-4]\nonumber\\
&+\frac{1}{(\xi_n-i)^2(1+\xi_n^2)^4}[96i\xi_n^3+80\xi_n^2+(64i-32i\xi_i^2)\xi_n+16\xi_i^2]\bigg].
\end{align}
Substituting (\ref{60}) into (\ref{56}), we get
\begin{align}\label{61}
\widetilde{\Phi}_5
=&\Big(\frac{323}{60}-\frac{31i}{60}\Big)\pi^2\sum_{j=1}^{n-1}X_jY_jh'(0)dx'+\Big(-\frac{5}{4}+\frac{103i}{32}\Big)X_nY_nh'(0)\pi\Omega_3dx'\nonumber\\
&+\Big(-\frac{7}{3}+\frac{49i}{24}\Big)X_nY_nh'(0)\pi^2dx'.
\end{align}
Let $X=X^T+X_n\partial_n,~Y=Y^T+Y_n\partial_n,$ then we have $\sum_{j=1}^{n-1}X_jY_j(x_0)=g(X^T,Y^T)(x_0).$
 Now $\widetilde{\Phi}$ is the sum of the $\widetilde{\Phi}_{(1,2,\cdot\cdot\cdot,5)}$. Combining with the five cases, this yields
\begin{align}\label{62}
\widetilde{\Phi}=\sum_{i=1}^5\widetilde{\Phi}_i
=&\Big[\Big(-\frac{3673}{8}-\frac{793i}{32}\Big)X_nY_n\Omega_3+\Big(-\frac{7}{3}+\frac{49i}{24}\Big)X_nY_n\pi\nonumber\\
&+\Big(-\frac{45937}{60}+\frac{169i}{60}\Big)\pi g(X^T,Y^T)\Big]\pi h'(0)dx'.
\end{align}
So, we are reduced to prove the following.
\begin{thm}\label{thmb1}
Let $M$ be a 4-dimensional compact manifold with boundary and $\nabla^{A}$ be an orthogonal
connection. Then we get the spectral Einstein functional associated to $\nabla_{X}^{A}\nabla_{Y}^{A}D_{A}^{-1}$
and $D_{A}^{-3}$ on compact manifolds with boundary
\begin{align}
\label{b2773}
&\widetilde{{\rm Wres}}[\pi^+(\nabla_{X}^{A}\nabla_{Y}^{A}D_{A}^{-1})\circ\pi^+(D_{A}^{-3})]\nonumber\\
=&\frac{4\pi^2}{3}\int_{M}\Big(Ric(V,W)-\frac{1}{2}sg(V,W)\Big) vol_{g}-2\int_{M}sg(V,W) vol_{g}\nonumber\\
&+\int_{\partial M}\Big[\Big(-\frac{3673}{8}-\frac{793i}{32}\Big)X_nY_n\Omega_3+\Big(-\frac{7}{3}+\frac{49i}{24}\Big)X_nY_n\pi\nonumber\\
&+\Big(-\frac{45937}{60}+\frac{169i}{60}\Big)\pi g(X^T,Y^T)\Big]\pi h'(0)vol_{\partial M},
\end{align}
where $s$ is the scalar curvature.
\end{thm}

\section*{Acknowledgements}
This work was supported by NSFC. 11771070. The authors thank the referee for his (or her) careful reading and helpful comments.

\end{document}